\documentclass[12pt]{amsart}
\usepackage{amscd,amssymb,amsthm,amsmath,amssymb,mathrsfs,enumerate,longtable,hyperref,url,mathrsfs,float}
\usepackage[all]{xy}
\usepackage[margin=2cm]{geometry}

\newtheorem{theorem}[equation]{Theorem}
\newtheorem*{theorem*}{Theorem}

\newtheorem{proposition}[equation]{Proposition}
\newtheorem*{proposition*}{Proposition}
\newtheorem{lemma}[equation]{Lemma}
\newtheorem{corollary}[equation]{Corollary}
\newtheorem*{corollary*}{Corollary}

\newtheorem*{problem*}{Problem}

\newtheorem*{question*}{Question}

\newtheorem*{construction*}{Construction}
\newtheorem*{techtheorem*}{Technical Theorem}
\newtheorem*{theoremAprime*}{Theorem A$^\prime$}
\newtheorem*{auxtheorem*}{Characterization Theorem}
\newtheorem*{maintheorem*}{Main Theorem}
\newtheorem*{conjecture*}{Conjecture}
\newtheorem*{fibrationprop*}{Fibration Theorem}
\newtheorem*{rattheorem*}{Rationality Theorem}

\theoremstyle{definition}

\newtheorem*{example*}{Example}

\newtheorem*{definition*}{Definition}
\newtheorem{remark}[equation]{Remark}
\newtheorem*{remark*}{Remark}

\makeatletter\@addtoreset{equation}{section} \makeatother

\author{Sebastián Fuentes}

\title{Almost simple subgroups of the space Cremona group}

\address{\emph{Sebastián Fuentes}\newline\textnormal{University of Edinburgh, Edinburgh, Scotland
\newline
\texttt{s.e.fuentes-olguin@sms.ed.ac.uk}}}

\newcommand{\mumu}{{\boldsymbol{\mu}}}
\newcommand\FF{\mathbf{F}}
\newcommand\CC{\mathbb{C}}
\newcommand\PP{\mathbb{P}}
\newcommand\QQ{\mathbb{Q}}

\newcommand{\Bir}{\operatorname{Bir}}

\newcommand{\OOO}{{\mathscr{O}}}

\newcommand{\Pic}{\operatorname{Pic}}

\newcommand\Aut{\operatorname{Aut}}
\newcommand\SL{\operatorname{SL}}
\newcommand\PSL{\operatorname{PSL}}
\newcommand\PGL{\operatorname{PGL}}

\newcommand\Cr{\operatorname{Cr}}
\newcommand\Sing{\operatorname{Sing}}

\newcommand\Bs{\operatorname{Bs}}
\newcommand\g{\operatorname{g}}

\newcommand{\Sym}{\operatorname{Sym}}

\usepackage[dvipsnames]{xcolor}
\usepackage{graphicx}

\usepackage[normalem]{ulem}

\begin{document}

\begin{abstract}
We prove that every finite almost simple subgroup that is not simple of the group of birational automorphisms of the three-dimensional complex projective space is isomorphic to one of the following: $\mathfrak{S}_5$, $\mathfrak{S}_6$, $\PGL_2(\FF_7)$ or $\PGL_2(\FF_9)$. We also prove that the group of birational automorphisms of the three-dimensional complex projective space contains unique (up to conjugation) subgroups isomorphic to $\mathrm{PGL}_2(\FF_7)$ and $\mathrm{PGL}_2(\FF_9)$.  
\end{abstract}

\maketitle
\setcounter{tocdepth}{1}

\tableofcontents

\section{Introduction}

Finite  simple non-abelian subgroups of the space Cremona group $\Cr_3(\CC)=\mathrm{Bir}(\mathbb{P}^3)$ have been classified by Yuri Prokhorov as follows.

\begin{theorem*}[\cite{P:JAG:simple}]
Let $G$ be a finite simple non-abelian subgroup of the group $\Cr_3(\CC)$. Then $G$ is isomorphic to one of the following groups:
\begin{center}
$\mathfrak{A}_5$, $\PSL_2(\FF_7)$, $\mathfrak{A}_6$, $\mathfrak{A}_7$, $\SL_2(\FF_8)$, $\mathrm{PSp}_{4}(\FF_3)$.
\end{center}
\end{theorem*}

Our aim is to extend this classification to finite almost simple subgroups of the group $\Cr_3(\CC)$. Recall that a finite group $G$ is said to be almost simple if
$$A\subseteq G\subseteq\Aut(A)$$
for a non-abelian simple group $A$. The group $A$ is called the socle of $G$. In this paper, we prove the following result.

\begin{maintheorem*}
Let $G$ be a finite almost simple subgroup of the group $\Cr_3(\CC)$ that is not simple. Then $G$ is isomorphic to one of the following groups:
\begin{center}
    $\mathfrak{S}_5,\quad$ $\mathfrak{S}_6,\quad$ $\PGL_2(\FF_7),\quad$ $\PGL_2(\FF_9)$.
\end{center}
Moreover, up to conjugation, the group $\Cr_3(\CC)$ contains unique subgroups isomorphic to $\PGL_2(\FF_7)$ and $\PGL_2(\FF_9)$.
\end{maintheorem*}

The key example for this theorem is the following.

\begin{example*}
Let $W$ be the smooth threefold 
\begin{equation}
\label{equation:W}
\big\{x_1y_1+x_2y_2+x_3y_3=0\big\}\subset\mathbb{P}^2_{x_1,x_2,x_3}\times\mathbb{P}^2_{y_1,y_2,y_3}.\tag{$\bigstar$}
\end{equation}
Then $\mathrm{Aut}(W)\simeq\mathrm{PGL}_3(\mathbb{C})\rtimes\mumu_2$ and $W$ is rational. Up to conjugation, the group $\mathrm{Aut}(W)$ contains unique subgroups isomorphic to $\PGL_2(\FF_7)$ and $\PGL_2(\FF_9)$. These actions give us subgroups in $\Cr_3(\CC)$ isomorphic to $\PGL_2(\FF_7)$ and $\PGL_2(\FF_9)$ (see Proposition \ref{proposition:W-actions} and Remark \ref{remark:PGL27-flag}).
\end{example*}

The next example shows that $\Cr_3(\CC)$ has subgroups isomorphic to $\mathfrak{S}_6$, so, in particular, it also contains subgroups isomorphic to $\mathfrak{S}_5$.

\begin{example*}
Let $X_1=\mathbb{P}^3$, let $X_2$ be the smooth quadric threefold in $\mathbb{P}^4$, let $X_3$ be the Segre cubic 
$$
\Big\{\sum_{i=1}^{6}x_i=\sum_{i=1}^{6}x^3_i=0\Big\}\subset\mathbb{P}^5,
$$
and let $X_{4}^\lambda$ be the quartic threefold given by 
$$
\Big\{\sum_{i=1}^{6}x_i=\sum_{i=1}^{6}x^4_i-\lambda\Big(\sum_{i=1}^{6}x^2_i\Big)^2=0\Big\}\subset\mathbb{P}^5,
$$
where $\lambda\in\{\frac{1}{2},\frac{1}{6},\frac{1}{7}\big\}$. Then each of these threefolds admits a faithful action of the group $\mathfrak{S}_6$, and all of them are rational \cite{CS2016,CKS}. The $\mathfrak{S}_6$-actions on these threefolds give $6$ subgroups in $\Cr_3(\CC)$ that are isomorphic to $\mathfrak{S}_6$. These subgroups are pairwise not conjugate \cite{cheltsov2014five,CheltsovShramovSelecta,BuBa,CKS}.
\end{example*}

Above, we described $6$ pairwise non-conjugate subgroups in $\mathrm{Cr}_3(\mathbb{C})$ that are isomorphic to $\mathfrak{S}_6$. We believe that these actions represent all the conjugacy classes of subgroups of $\Cr_3(\CC)$ isomorphic to $\mathfrak{S}_6$.

\begin{conjecture*}
$\mathrm{Cr}_3(\mathbb{C})$ contains exactly $6$ conjugacy classes of subgroups isomorphic to $\mathfrak{S}_6$.
\end{conjecture*}

In the following we outline the proof of the Main Theorem and the reductions we will use. Let $X$ be a rational threefold such that $\Aut(X)$ contains a finite almost simple group $G$ that is not simple. Then $G$ contains a non-abelian simple subgroup $A$ such that $G$ is isomorphic to a subgroup of $\mathrm{Aut}(A)$. From \cite{P:JAG:simple}, we know that $A$ is isomorphic to one of the following groups:
\begin{center}
$\mathfrak{A}_5$, $\PSL_2(\FF_7)$, $\mathfrak{A}_6$, $\mathfrak{A}_7$, $\SL_2(\FF_8)$, $\mathrm{PSp}_{4}(\FF_3)$.
\end{center}

Since $\mathrm{Aut}(\mathfrak{A}_5)\simeq\mathfrak{S}_5$, we may assume that $A\not\simeq\mathfrak{A}_5$. Then, taking a $G$-equivariant resolution of singularities of $X$ and applying the $G$-equivariant Minimal Model Program, we may assume that $X$ has terminal singularities, every $G$-invariant Weil divisor on $X$ is $\mathbb{Q}$-Cartier, and there exists a $G$-equivariant dominant morphism $\pi\colon X\to Z$ such that $Z$ is normal, $\pi$ has connected fibers, and one of the following three possibilities holds:
\begin{enumerate}
\item $Z$ is a point and $X$ is a $G\mathbb{Q}$-Fano threefold, i.e. $\mathrm{rk}\,\mathrm{Cl}(X)^G=1$;
\item $Z=\mathbb{P}^1$ and the general fiber of $\pi$ is a smooth del Pezzo surface;
\item $Z$ is a rational surface, and general fiber of $\pi$ is isomorphic to $\mathbb{P}^1$.
\end{enumerate}
If $A$ is isomorphic to one of the groups $\mathfrak{A}_7$, $\SL_2(\FF_8)$, $\mathrm{PSp}_{4}(\FF_3)$, it follows from \cite{P:JAG:simple,Bea11,CheltsovPinardinProkhorov} that $Z$ is a point, and one of the following cases holds:
\begin{itemize}
\item $A\simeq\mathfrak{A}_7$ and $X\simeq\mathbb{P}^3$;
\item $A\simeq\mathrm{PSp}_{4}(\FF_3)$ and $X\simeq\mathbb{P}^3$;
\item $A\simeq\mathrm{PSp}_{4}(\FF_3)$ and $X$ is isomorphic to the Burkhardt quartic in $\mathbb{P}^4$;
\item $A\simeq \SL_2(\FF_8)$ and $X$ is isomorphic to the unique smooth Fano threefold of Picard rank $1$ and genus $7$ that is described in \cite[Example 2.11]{P:JAG:simple}.
\end{itemize}
In each of these cases, the group $A$ is a maximal finite subgroup of $\mathrm{Aut}(X)$, so $G=A$, which contradicts our assumption that $G$ is not simple. Hence, either $A\simeq\PSL_2(\FF_7)$ or $A\simeq\mathfrak{A}_6$. But
\begin{align*}
\mathrm{Aut}(\PSL_2(\FF_7))&\simeq\PGL_2(\FF_7),\\
\mathrm{Aut}(\mathfrak{A}_6)&\simeq\mathrm{P\Gamma L}_2(\FF_9).
\end{align*}

Therefore, if $A\simeq\PSL_2(\FF_7)$, then $G\simeq\PGL_2(\FF_7)$. On the other hand, if $A\simeq\mathfrak{A}_6$, then one of the following $4$ cases holds:
\begin{enumerate}
\item $G\simeq\mathfrak{S}_6\simeq\mathfrak{A}_6\rtimes\mumu_2$;
\item $G\simeq\mathfrak{A}_6\rtimes\mumu_2\simeq\PGL_2(\FF_9)$;
\item $G\simeq\mathfrak{A}_6.\mumu_2\simeq M_{10}$; 
\item $G\simeq\mathrm{P\Gamma L}_2(\FF_9)$.
\end{enumerate}

After the previous discussion, it is clear that Main Theorem will follow directly from the two theorems stated below.  

\begin{techtheorem*}
    \label{techtheorem}
    The group $\Cr_3(\CC)$ contains no subgroup isomorphic to $M_{10}$.
\end{techtheorem*}

\begin{auxtheorem*}
    \label{auxtheorem}
    Up to conjugation, the group $\Cr_3(\CC)$ contains unique subgroups isomorphic to $\PGL_2(\FF_7)$ and $\PGL_2(\FF_9)$.
\end{auxtheorem*}

Now we restate and simplify these theorems in geometric language. In the following, we assume that $G\simeq\PGL_2(\FF_7)$, $G\simeq\PGL_2(\FF_9)$ or $G\simeq M_{10}$. First, we prove

\begin{fibrationprop*}
    \phantomsection
    \label{theorem:fibration}
    Let $X$ be a rationally connected threefold with terminal singularities that admits a faithful action a group $G$ isomorphic to one of the groups $\PGL_2(\FF_7),\PGL_2(\FF_9),M_{10}$. Then $X$ is $G$-birational  neither to a $G$-conic bundle nor to a $G$-del Pezzo fibration.
\end{fibrationprop*}

In view of the \hyperref[theorem:fibration]{Fibration Theorem} we can assume that $X$ is a $G\mathbb{Q}$-Fano threefold (see \cite{P:GFano1}). Finally, using the technique developed in \cite{Prokhorov-p-groups,P:JAG:simple,Prokhorov-2-groups,YuraCostya-p-groups,Kuznetsova,Loginov,YuraS6,CheltsovPinardinProkhorov}, we prove the following statements.

\begin{techtheorem*}[Geometric version]
    \label{tech}
    Let $X$ be a rational $G\QQ$-Fano threefold. Then $G~\not\simeq~M_{10}$. 
\end{techtheorem*}

\begin{auxtheorem*}[Geometric version]
    \label{char}
    Let $X$ be a rational $G\QQ$-Fano threefold such that $G\simeq\PGL_2(\FF_7)$ or $G\simeq\PGL_2(\FF_9)$. Then $X$ is $G$-birational to the threefold~\eqref{equation:W}.
\end{auxtheorem*}

In fact, our arguments give a stronger result for rationally connected threefolds. In this setting, additional irrational models occur, beginning with the following example.

\begin{example*}[Symmetric sextic threefold]
    Let $X_6$ be the symmetric sextic Fano threefold
    \begin{equation}
        \label{equation:sextic}
        X_6=\left\{\sum_{i=0}^6x_i=\sum_{i=0}^6x_i^2=\sum_{i=0}^6x_i^3=0\right\}\subset\PP^6
    \end{equation}
    It is proved in \cite{BeauvilleA7} that $X_6$ is irrational. Moreover, in \cite{YuraS6} the author proves that $X_6$ is the unique birational model for an $\mathfrak{S}_7$-action, up to conjugation.
\end{example*}

Two further models arise for $G\simeq\PGL_2(\FF_7)$. Using \textsc{MAGMA}, we found exactly two $\PGL_2(\FF_7)$-invariant complete intersections of a quadric and a cubic hypersurface in $\PP^5$ and verified their smoothness. We denote them by
$$X_{2,3}^{(1)},\qquad X_{2,3}^{(2)}\subset\PP^5.$$

\begin{rattheorem*}
The threefolds $X_{2,3}^{(1)}$ and $X_{2,3}^{(2)}$ are irrational.
\end{rattheorem*}

This corresponds to Theorem \ref{theorem:X-2-3-irrational}, and the proof uses the Clemens--Griffiths rationality criterion and the induced action of $\PGL_2(\FF_7)$ on the intermediate Jacobian. Since $X_{2,3}^{(1)}$ and $X_{2,3}^{(2)}$ are smooth Fano threefolds, they are rationally connected. Thus, although they do not give subgroups of the space Cremona group, they must be retained when one studies birational actions on rationally connected threefolds.  

\begin{remark*}
    Another irrational threefold shows up in the proof of Lemma~\ref{lemma:PGL27-gorenstein}, but it does not contribute an additional model to the Cremona classification. Let $A=\PSL_2(\FF_7)$. It is shown in \cite{GP11} that there exists a unique $A$-invariant divisor $X$ of bidegree $(2,2)$ in $\PP_{x_1,x_2,x_3}^2\times\PP_{y_1,y_2,y_3}^2$, given by the equation
$$x_1x_2y_1^2+x_1^2y_1y_2+x_2x_3y_2^2+x_3^2y_1y_3+x_2^2y_2y_3+x_1x_3y_3^2=0,$$
and $X$ is irrational. One sees that the discriminant of the first coordinate projection is a sextic curve in $\PP^2$ that is $A$-invariant, so $[\Aut(X):A]\leqslant2$. On the other hand, the symmetry of the equation gives that the involution of $\PP^2\times\PP^2$ swapping both factors induces an involution $\tau\in\Aut(X)$, and then $[\Aut(X):A]=2$. However, since $A$ acts diagonally on $\PP^2\times\PP^2$, we get that $\Aut(X)\simeq A\times\langle\tau\rangle$. In particular, $X$ does not admit an action of $\PGL_2(\FF_7)$.
\end{remark*}

For rationally connected threefolds, \cite[Theorem~1.5]{P:JAG:simple} also allows the additional socle $A\simeq\PSL_2(\FF_{11})$. Its two models are the Klein cubic and a smooth prime Fano threefold of genus~$8$, both with full automorphism group~$A$. For the Klein cubic this is stated in \cite[Example~2.6]{P:JAG:simple}. In the case of the smooth prime Fano threefold of genus~$8$, this follows from \cite[Lemma~4.3.4 and Proposition~B.6.1]{KPS:Hilb} with \cite[Theorem~2.1]{Roulleau:Klein}. By \cite[Theorem~4.3]{BCDP:quasisimple}, every birational transformation normalizing an $A$-action is biregular. Hence neither action extends to a larger almost simple group. 

Combining our arguments with these observations and \cite[Theorem~1.5]{P:JAG:simple}, we obtain the following classification.

\begin{corollary*}[Rationally connected threefolds]
Let $X$ be a rationally connected threefold, and let $G\subset\Bir(X)$ be a finite almost simple subgroup that is not simple. Then $G$ is isomorphic to one of the following groups:
$$
\mathfrak{S}_5,\qquad
\mathfrak{S}_6,\qquad
\mathfrak{S}_7,\qquad
\PGL_2(\FF_7),\qquad
\PGL_2(\FF_9).
$$
All the possibilities occur. Furthermore, if $G$ is not isomorphic to $\mathfrak{S}_5$ or $\mathfrak{S}_6$, then $(X,G)$ is $G$-birational to one of the following models:
\begin{enumerate}
\item $\mathfrak{S}_7$ acting on the symmetric sextic Fano threefold \eqref{equation:sextic}.

\item $\PGL_2(\FF_7)$ acting on the threefold \eqref{equation:W};

\item $\PGL_2(\FF_7)$ acting on one of the two smooth complete intersections
$$
X_{2,3}^{(1)},\quad X_{2,3}^{(2)}\subset\PP^5;
$$
\item $\PGL_2(\FF_9)$ acting on the threefold \eqref{equation:W}.
\end{enumerate}
\end{corollary*}

In Section~\ref{section:m10} we collect a series of lemmas about geometric actions of the groups $\mathfrak{A}_6$ and $M_{10}$ on curves, surfaces and threefolds, and we prove the~\hyperref[theorem:fibration]{Fibration Theorem}. We also describe the actions of the groups $\PGL_2(\FF_7)$ and $\PGL_2(\FF_9)$ on the threefold \eqref{equation:W}. Sections~\ref{section:A6-ext-gorenstein} and \ref{section:A6-ext-non-Gorenstein} are devoted to the proof of the \hyperref[techtheorem]{Technical Theorem}, under the assumption that $X$ has Gorenstein and non-Gorenstein singularities, respectively. In Section~\ref{section:chartheorem} we prove the \hyperref[auxtheorem]{Characterization Theorem}. The proof of this theorem is presented in a simplified way to avoid repetitions of arguments. For the group $G=\PGL_2(\FF_9)$ the proof is essentially the same as for the \hyperref[techtheorem]{Technical Theorem}, taking into account a small difference between the representation theory of these groups. In the case of the group $G=\PGL_2(\FF_7)$, the proof will follow the same structure as the previous cases, but there will be some major differences that are discussed in Subsections~\ref{subsection:PGL27-gorenstein} and \ref{subsection:PGL27-non-gorenstein}.

\medskip
\noindent
\textbf{Notations.}
Throughout this paper, we employ the~following notations and assumptions:
\begin{itemize}
\item $\mumu_n$ denotes the~cyclic group of order $n$;
\item $\mathfrak{D}_n$ denotes the~dihedral group of order $2n$;
\item $\mathfrak{S}_n$ and $\mathfrak{A}_n$ are the~symmetric and alternating groups, respectively;
\item all varieties are assumed to be projective, normal, irreducible and defined over $\mathbb{C}$ unless mentioned otherwise;
\item if $X$ is a variety, we denote by $\rho(X)$ the~rank of the~group $\mathrm{Pic}(X)$, and we denote by $\mathrm{r}(X)$ the rank of the group $\mathrm{Cl}(X)$;
\item if $X$ is a variety and $G$ is a~finite subgroup of $\mathrm{Aut}(X)$, we denote by $\rho^G(X)$ the~rank of the~group $\mathrm{Pic}(X)^G$, and we denote by $\mathrm{r}^G(X)$ the rank of the group $\mathrm{Cl}(X)^G$;
\item a~variety $X$ is said to be a~$G\mathbb{Q}$-Fano variety for a~finite subgroup $G\subset\mathrm{Aut}(X)$ if the~following $3$ conditions are satisfied: $X$ has terminal singularities, $-K_X$ is ample, and $\mathrm{r}^G(X)=1$;
\item if $X$ is a Fano threefold with canonical Gorenstein singularities, then the Fano index $\iota(X)$ is the largest integer $r$ such that $-K_X\sim rH$ for some $H\in\mathrm{Pic}(X)$, and the genus of $X$ is the positive integer 
$$\g(X)=\mathrm{dim}\big(|-K_X|\big)-1=\frac{(-K_X)^3+2}{2};$$
\item a projective threefold $X$ is said to be almost Fano if: $X$ has at most terminal Gorenstein singularities, $-K_X$ is big and nef, and the natural morphism
$$\phi:X\to \operatorname{Proj}\left(\bigoplus_{m\geqslant0}H^0(X,\mathscr{O}_X(-mK_X))\right)$$
does not contract divisors.
\item if $C$ is an irreducible curve, its arithmetic genus is denoted by $\mathrm{p}_a(C)$;
\item if $C$ is a smooth irreducible curve, its genus is denoted by $g(C)$.
\end{itemize}

\medskip
\noindent

\textbf{Acknowledgements.} I am very grateful to Ivan Cheltsov for introducing me to the topic, his attentive guidance and constant advices during my study of the subject and the creation of this document. I also thank Zhijia Zhang for informing me of the existence of two smooth $\PGL_2(\FF_7)$-invariant complete intersections of a quadric and a cubic in $\PP^5$, which he found using MAGMA.

\textbf{AI disclosure.} The AI model Chat-GPT 6 Astra by OpenAI has been used for proofreading. All the arguments and ideas presented here come from the author and the works cited throughout this document.

\section{Geometry of the Mathieu group $M_{10}$}
\label{section:m10}

Throughout this section, $G$ will denote a group isomorphic to $M_{10}$, while $A$ will be its unique subgroup isomorphic to $\mathfrak{A}_6$.

\subsection{Actions of the group $\mathfrak{A}_6$}
\label{subsection:A6}

\begin{lemma}
    \label{lemma:A6-genus}
    Let $C$ be a~smooth irreducible curve that admits a faithful action of the group $A$. Then the length of an $A$-orbit in $C$ can be $72,90,120,180,360$, $C$ is not hyperelliptic, and if $g=g(C)\leqslant30$ then $g\in\{10,16,19,25\}$.
\end{lemma}

\begin{proof}
    The stabilizer of $P\in C$ has to be a cyclic subgroup of $A$, so the length of an $A$-orbit can be just $72,90,120,180,360$ (see \cite{lmfdb}). Besides, $C$ cannot be hyperelliptic as $A$ does not embed into $\PGL_2(\CC)$. It is also known that $G$ does not act faithfully on an elliptic curve, so we can assume that $2\leqslant g\leqslant30$. Since $C\rightarrow C/A$ is a finite cover of degree $\vert A\vert$, if $m_1,\ldots,m_r$ denote the ramification indices of it and $h=g(C/A)$, by the Riemann-Hurwitz formula we have that 
    $$2g-2=\vert A\vert \left(2h-2+\sum_{i=1}^r\left(1-\frac{1}{m_i}\right)\right).$$
    Then if $a_m$ denotes the number of branch points of index $m$ we get that
    $$g-1=360(h-1)+90 a_2+120a_3+135a_4+144a_5.$$
    Analyzing the possible solutions for the restriction $2\leqslant g\leqslant30$ we get the claimed list of possible values for $g$.
\end{proof}

\begin{lemma}
\label{lemma:A6-singular-curves}
Let $C$ be an irreducible~singular curve with at most locally planar singularities that admits a faithful action of the~group $A$. Then \mbox{$\mathrm{p}_a(C)\geqslant 46$}. 
\end{lemma}

\begin{proof}
    Let $\pi:\widetilde{C}\to C$ be its normalization. By Lemma \ref{lemma:A6-genus} $g(\widetilde{C})\geqslant10$ and it is known that $\mathrm{p}_a(C)\geqslant g(\widetilde{C})+\vert\Sing(C)\vert$. By assumption, if $P\in C$ is a singular point the Zariski tangent space $T_{C,P}$ is a $2$-dimensional faithful representation of its stabilizer. On the other hand the groups, by the subgroup structure of $A$ and the character tables of its subgroups, every subgroup admitting a faithful $2$-dimensional complex representation has index at least $36$. This means that $\vert\Sing(C)\vert\geqslant36$ and the assertion follows.
\end{proof}

\begin{lemma}
\label{lemma:K3-a6}
Let $S$ be a smooth K3 surface that admits a faithful action of the group $A$. Then $\mathrm{Pic}(S)^{\mathfrak{A}_6}\simeq\mathbb{Z}$.
\end{lemma}

\begin{proof}
    Consider the quotient $\overline{S}=S/A$, so we have $\rho(S)^{A}=\rho(\overline{S})$. Since $A$ is simple and non-abelian, the action of $A$ is symplectic, and from \cite[Theorem 3]{Xiao} we know that $\overline{S}$ is a K3 surface with $7$ Du Val singular points, corresponding to $1$ singular point of type $A_1$ and $2$ singular points of each type $A_2,A_3,A_4$. Consider $\psi:\widetilde{S}\to\overline{S}$ the minimal resolution of singularities. Then we have that
    $$\rho(\widetilde{S})=\rho(\overline{S})+2\cdot 4+2\cdot 3+2\cdot 2+1=\rho(\overline{S})+19.$$
    The conclusion is obtained from the fact that $\Pic(\widetilde{S})$ is torsion-free and verifies $\rho(\widetilde{S})\leqslant20$.
\end{proof}

\begin{corollary}
\label{corollary:K3-a6}
Let $S$ be a K3 surface with at most Du Val singularities that admits a faithful action of the group $A$. Then $S$ is smooth.
\end{corollary}

\begin{theorem}[{\cite[Lemma B.3]{cheltsov2014five}}]
\label{theorem:A6-dP}
    Let $S$ be a smooth del Pezzo surface such that $A\subset\Aut(S)$. Then $S\simeq\PP^2$.
\end{theorem}

\begin{lemma}
    \label{lemma:A6-invariant-curves-P2}
    Let $C\subset\PP^2$ be an $A$-invariant plane curve of $\deg(C)<12$. Then $\deg(C)=6$.
\end{lemma}

\begin{proof}
    An action of $A$ on $\PP^2$ comes from a $3$-dimensional representation $V$ of the Valentiner group $3.\mathfrak{A}_6$. Then an $A$-invariant curve of degree $k$ in $\PP^2$ will correspond to a $1$-dimensional subrepresentation of $\Sym^kV$. One can compute using MAGMA that for $1\leqslant k<12$ this just happens when $k=6$ and such a representation is unique. 
\end{proof}

We will also present the low degree $A$-invariant curves in $\PP^3$.

\begin{lemma}[{\cite{cheltsov2014five}}]
    \label{lemma:A6-invariant-curves}
    Let $A$ act on $\PP^3$ by a $4$-dimensional irreducible representation $V$ of the covering group $2.\mathfrak{A}_6$, and let $C\subset\PP^3$ be an $A$-invariant curve of $\deg(C)\leqslant6$. Then $\deg(C)=6$ and $C$ consists of $6$ disjoint lines. Moreover, $\PP^3$ contains exactly two such curves.
\end{lemma}

To finish this subsection we present some results about actions on threefolds.

\begin{theorem}[{\cite[Lemma B.2]{cheltsov2014five}}]
\label{theorem:A6-non-fixed}
    Let $X$ be a threefold with terminal singularities such that $A\subset\Aut(X)$. Then $X$ contains no $A$-fixed points. 
\end{theorem}

\begin{lemma}
    \label{lemma:A6-bpf}
    Let $X$ be a Fano threefold with canonical Gorenstein singularities such that $A\subset\Aut(X)$. Then $\Bs(\vert-K_X\vert)=\varnothing$. 
\end{lemma}

\begin{proof}
    Suppose that the base locus of $\vert-K_X\vert$ is not empty. By \cite{Shi89}, one of the following holds:
    \begin{enumerate}
        \item $\Bs(\vert-K_X\vert)\simeq\PP^1$, and it is contained in the smooth locus of $X$.
        \item $\Bs(\vert-K_X\vert)$ consists of a single point which is a hypersurface cDV singularity and is an ordinary double point of a general member of $\vert-K_X\vert$.  
    \end{enumerate}
    It is clear that $\Bs(\vert-K_X\vert)$ is $A$-invariant. Now, if $C:=\Bs(\vert-K_X\vert)\simeq\PP^1$ the action of $A$ on $C$ is trivial, but this is not possible as $A$ does not have faithful $3$-dimensional linear representations. If $\Bs(\vert-K_X\vert)=\{P\}$, a general member $S\in\vert-K_X\vert$ is such that $\dim T_{S,P}=3$. Since $S$ is Cartier, we get that
    $$\dim T_{X,P}\leqslant\dim T_{S,P}+1=4$$
    but $A$ also does not have faithful $4$-dimensional linear representations. We conclude by contradiction.
\end{proof}

\begin{lemma}
    \label{lemma:A6-hyperelliptic}
    Let $X$ be a Fano threefold with canonical Gorenstein singularities such that $A\subset\Aut(X)$. Assume $-K_X$ is not very ample. Then $X$ is isomorphic to 
    $$\left\{\sum_{i=1}^6x_i=0,\quad\sum_{i=1}^6x_i^2=0,\quad w^2=\sum_{i=1}^6x_i^4\right\}\subset\PP(1^6,2),$$
    and its anticanonical morphism is a double cover of a smooth quadric threefold.
\end{lemma}

\begin{proof}
    Put $g=\frac{1}{2}(-K_X)^3+1$ and $V=H^0(X,\OOO_X(-K_X))^\vee$. By Lemma~\ref{lemma:A6-bpf} and \cite[Theorem~2.12]{HT}, the anticanonical linear system gives an $A$-equivariant double cover
    \[
        \phi:X\longrightarrow Y\subset\PP(V),
    \]
    where $Y$ is a threefold of minimal degree $\deg(Y)=g-1$. The kernel of the action on $Y$ has order at most two, so simplicity of $A$ implies that this action is faithful. Thus $V$ is a faithful representation of dimension $g+2$, and $g\geqslant3$.

If $g=3$, then $V$ is irreducible and $Y$ is a smooth quadric threefold. The double cover is branched along a member of $\vert\OOO_Y(4)\vert$, so \cite[6.5]{DoubleQuadrics} gives the threefold in the statement. If $g\geqslant4$, from the classification of varieties of minimal degree we get the following possibilities:
    \begin{enumerate}
        \item $Y$ is the Veronese cone in $\PP^6$, in particular $Y\simeq\PP(1,1,1,2)$;
        \item $Y$ is a possibly singular rational normal scroll
    \end{enumerate}
    If $Y$ is the Veronese cone, its vertex corresponds to an invariant line in $V$, so we have a decomposition $V=W\oplus L$ where $\dim(W)=6$. Now projection from the vertex gives an $A$-invariant Veronese surface $S=\nu_2(\PP^2)\subset\PP(W)$ and $A$ acts faithfully on $S$, otherwise the action of $A$ on $V$ would be trivial as $S$ spans $\PP(W)$. Note that $W=\mathbf{1}\oplus W_5$ splits as a trivial representation and a $5$-dimensional faithful $A$-representation, so the invariant hyperplane $\PP(W_5)\subset\PP(W)$ cuts out an $A$-invariant divisor on $S$ whose pullback by the Veronese embedding $\nu_2:\PP^2\hookrightarrow\PP^5$ has degree $2$. This is not possible as it contradicts Lemma \ref{lemma:A6-invariant-curves-P2}.

    In the scroll case, the ruling planes form a family parametrized by $\PP^1$, preserved by $A$. Since $A$ acts trivially on $\PP^1$, every ruling plane is $A$-invariant. Its corresponding three-dimensional subrepresentation of $V$ is trivial, so the plane is pointwise fixed. As these planes cover $Y$, this contradicts faithfulness of the action on $Y$.
\end{proof}

\subsection{Extensions of $\mathfrak{A}_6$}
\label{subsection:W-actions}

Recall that $\PGL_2(\FF_9)$ and $M_{10}$ are, respectively, split and non-split extensions of $A=\mathfrak{A}_6$ by $\mumu_2$. Together with $\mathfrak{S}_6$ they correspond to the non-trivial elements of $\operatorname{Out}(A)=\mumu_2^2$. In this subsection we determine which of these extensions act faithfully on the threefold \eqref{equation:W}.

It is known that every faithful action of $A$ on $\PP^2$ is obtained as the projectivization of one of the four faithful irreducible $3$-dimensional representations of the Valentiner group $\widehat{A}=3.A$. Every automorphism of $A$ lifts uniquely to $\widehat{A}$: lift it to the universal central extension $6.A$, quotient by its characteristic central subgroup of order $2$, and use perfectness for uniqueness. Thus $\Aut(\widehat{A})~\simeq~\Aut(A)$, and $\operatorname{Out}(A)$ acts on isomorphism classes of $\widehat{A}$-representations by $[\alpha]\cdot[\rho]=[\rho\circ\widetilde{\alpha}^{-1}]$.

\begin{lemma}
\label{lemma:3A6-transitive}
The group $\operatorname{Out}(A)$ acts transitively on the four isomorphism classes of faithful irreducible $3$-dimensional representations of $\widehat{A}$. The class corresponding to $\PGL_2(\FF_9)$ acts by duality.
\end{lemma}

\begin{proof}
For each nontrivial central character of $\widehat{A}$, there are exactly two three-dimensional irreducible representations. Their character values on the two conjugacy classes of order \(5\) are distinct real numbers, interchanged between the two representations; see \cite{lmfdb}.
The outer class corresponding to $M_{10}$ fixes the center and interchanges the two order $5$ classes. It therefore exchanges the two representations with the same central character. The class corresponding to $\PGL_2(\FF_9)$ inverts the center and preserves each order $5$ class. It therefore acts by duality, which inverts the central character and preserves these real character values. See \cite[Table 6.3.1]{SimpleGrps} and the character tables.
\end{proof}

Now let $H\simeq A$ be a subgroup of $\Aut(W)\simeq\PGL_3(\CC)\rtimes\mumu_2$. Since $A$ has no quotient of order $2$, its action arises from a faithful three-dimensional representation $\rho:\widehat{A}\to\SL(V)$ and the twisted diagonal action on
$$W=\{([v],[\ell])\in\PP(V)\times\PP(V^\vee)\mid\ell(v)=0\}.$$
Perfectness and Schur's lemma show that $W$ is the unique $H$-invariant divisor of bidegree $(1,1)$.

\begin{lemma}
\label{lemma:normalizer-2}
One has $N_{\Aut(W)}(H)\simeq\PGL_2(\FF_9)$.
\end{lemma}

\begin{proof}
Write $N=N_{\Aut(W)}(H)$. It is known that $C_{\PGL(V)}(H)=1$. Then  a factor-preserving element of $N$ induces an automorphism whose lift fixes $[\rho]$, since $\widehat{A}$ has no non-trivial characters. Lemma~\ref{lemma:3A6-transitive} therefore gives $N\cap\PGL(V)=H$.

Choose an involution $\mu\in\Aut(A)$ representing the class of $\PGL_2(\FF_9)$. Since $\rho\circ\widetilde{\mu}\simeq\rho^\vee$, coordinates can be chosen so that
$$a\cdot(x,y)=(\overline{\rho}(a)x,\overline{\rho}(\mu(a))y).$$
Then $\tau(x,y)=(y,x)$ normalizes $H$, induces $\mu$, and preserves $W$ by uniqueness of the invariant $(1,1)$-divisor. Hence $N=H\rtimes\langle\tau\rangle\simeq\PGL_2(\FF_9)$.
\end{proof}

We obtain the following result.

\begin{proposition}
\label{proposition:W-actions}
The only extension of the form $\mathfrak{A}_6.\mumu_2$ admitting a faithful action on $W$ is $\PGL_2(\FF_9)$. Moreover, all subgroups of $\Aut(W)$ isomorphic to $\PGL_2(\FF_9)$ are conjugate.
\end{proposition}

\begin{proof}
Every such extension equals the normalizer of its index-two subgroup $H\simeq A$ by Lemma~\ref{lemma:normalizer-2} and comparison of orders. For two corresponding representations $\rho_i:\widehat{A}\to\SL(V_i)$, Lemma~\ref{lemma:3A6-transitive} gives $\alpha\in\Aut(\widehat{A})$ and $T:V_1\xrightarrow{\sim}V_2$ with $T\rho_1(h)T^{-1}=\rho_2(\alpha(h))$. The isomorphism
$$([v],[\ell])\longmapsto([Tv],[\ell\circ T^{-1}])$$
conjugates the corresponding subgroups and hence their normalizers.
\end{proof}

\begin{remark}
\label{remark:PGL27-flag}
The same argument applies to $A=\PSL_2(\FF_7)$. Its Schur multiplier has order $2$, so its embeddings into $\PGL_3(\CC)$ lift to $\SL_3(\CC)$. Its two faithful irreducible three-dimensional representations are dual and are interchanged by an involutive outer automorphism; see \cite{lmfdb,GroupNames}. Thus the induced $A$-action on $W$ extends to $N_{\Aut(W)}(A)\simeq\PGL_2(\FF_7)$, and these subgroups form a single conjugacy class in $\Aut(W)$.
\end{remark}

\subsection{Actions of the group $M_{10}$}

\begin{lemma}
\label{lemma:A6-ext-projective-spaces}
For $n\in\{2,3,4,5\}$, the group $\mathrm{PGL}_n(\CC)$ has no subgroup isomorphic to $G$.
\end{lemma}

\begin{proof}
    The ordinary irreducible character degrees of $G$ are $1,1,9,9,10,10,10,16$, so such an embedding has to come from a central extension of $G$. The Schur multiplier of $G$ is cyclic of order $3$, so every projective representation of $G$ lifts to a linear representation of its covering group $\widehat{G}=3.M_{10}$, with $\mumu_3$ acting by scalars. From \cite{lmfdb} one sees that an irreducible representation of $\widehat{G}$ with non-trivial central character has dimension at least $6$.
\end{proof}

\begin{lemma}
\label{lemma:A6-ext-curves}
Let $C$ be a~smooth curve that admits a faithful action of the group $G$. Then the~$G$-orbits in $C$ are of length $90,144,180,240,360,720$. Moreover, the~curve $C$ is not hyperelliptic and if $g(C)\leqslant 48$, then $g(C)=46$.
\end{lemma}

\begin{proof}
    The proof is essentially the same as that of Lemma \ref{lemma:A6-genus}.
\end{proof}

\begin{corollary}
\label{corollary:A6-ext-singular-curves}
Let $C$ be an irreducible~singular curve with at most locally planar singularities that admits a faithful
action of the~group $G$. Then $\mathrm{p}_a(C)\geqslant 86$.
\end{corollary}

\begin{lemma}
    \label{lemma:A6-ext-surfaces}
    Let $S$ be a smooth surface admitting a faithful action of $G$. Then $S$ does not have $G$-orbits of length less than $36$.
\end{lemma}

\begin{proof}
    For $P\in S$, the stabilizer $G_P$ has an irreducible $2$-dimensional tangent representation. 
\end{proof}

\begin{lemma}[{\cite{IgorVasya}}]
\label{lemma:A6-ext-cr2}
    The group $M_{10}$ does not admit an embedding into $\Cr_2(\CC)$. In particular, no rational surface admits a faithful action of $M_{10}$. 
\end{lemma}

\begin{lemma}[{\cite[Corollary 2.5.6]{CheltsovPinardinProkhorov}}]
    \label{Lemma:A6-ext-reduced-divisor}
    Let $S$ be a smooth surface such that its automorphism group contains a subgroup isomorphic to $A$. If $D\in\vert-K_S\vert$ is $A$-invariant, then $D$ is not reduced.
\end{lemma} 

\begin{lemma}
\label{lemma:A6-ext-fixed}
Let $X$ be a threefold with terminal singularities admitting a faithful action of a
group $\Gamma\simeq\operatorname{PSU}_3(\FF_2)$. Then $X$ has no $\Gamma$-invariant points.
\end{lemma}

\begin{proof}
    Suppose $P\in X$ is a $\Gamma$-fixed point. Since $\Gamma$ does not have faithful representations of dimension less than $8$, we see that $P$ has to be a non-Gorenstein singular point. Let $r$ be the Gorenstein index of the singularity $(X,P)$ and consider $\pi:X^\prime\to X$ its index-1 cover with $P^\prime=\pi^{-1}(P)$. Then, there is an exact sequence of groups 
    $$1\longrightarrow\mumu_r\longrightarrow\Gamma^\prime\xrightarrow{\;\;p\;\;} \Gamma\longrightarrow1$$
    where $\Gamma^\prime$ is a finite subgroup of automorphisms of the germ $(X^\prime,P^\prime)$ that fixes the point $P^\prime$ and acts faithfully on the tangent space $T_{X^\prime,P^\prime}$. In this setting, the group $\Gamma$ acts by conjugation on the group $\mumu_r$, and this action gives a morphism of groups $\Phi:\Gamma\to\Aut(\mumu_r)$. Write $N=\mumu_3^2\triangleleft\Gamma$ and recall that
    \begin{center}
        $\Gamma\simeq N\rtimes Q_8,\qquad [\Gamma,\Gamma]\simeq N\rtimes\mumu_2,\qquad\Gamma^{\mathrm{ab}}\simeq\mumu_2^2,$
    \end{center}
    and $Z(\Gamma)=1$. Moreover, every linear representation of $\Gamma$ of dimension less than $8$ is trivial on $N$. Since $\Aut(\mumu_r)$ is abelian, $\Phi$ factors through $\Gamma^{\mathrm{ab}}$, so its image is a $2$-group. 

    In the rest of the proof, we will make use of the following observation. Suppose that the tangent space $T_{X^\prime,P^\prime}$ decomposes into $\Gamma^\prime$-invariant $\mumu_r$-eigenspaces of dimension at most $2$. On each two-dimensional eigenspace,
    projectivization gives a homomorphism $\Gamma\to\PGL_2(\CC)$, whose restriction to $N$ is trivial. Indeed, $N$ is a minimal normal subgroup of $\Gamma$, while $\PGL_2(\CC)$ contains no subgroup isomorphic to $\mumu_3^2$. Consequently, $p^{-1}(N)$ acts by scalars on every summand and it centralizes $\Gamma^\prime$ on the tangent space. Since the action of $\Gamma^\prime$ on $T_{X^\prime,P^\prime}$ is faithful, this implies $p^{-1}(N)\subset Z(\Gamma^\prime)$, and hence $N\subset Z(\Gamma)$, a contradiction.

    Assume first that $r=2$. By the classification of terminal threefold singularities \cite{Reid:YPG}, a generator $z$ of $\mumu_2$ has a $(-1)$-eigenspace $T_{-1}$ of dimension 3 and a trivial eigenspace $T_1$ of dimension $0$ or $1$. Both are $\Gamma^{\prime}$-invariant as $\mumu_2$ is central. Note that $\operatorname{Sym}^{2}(T_{-1})$ is a $6$-dimensional representation of $\Gamma\simeq\Gamma^{\prime}/\mumu_2$, so $N$ acts trivially on it. Then we see that every element of $p^{-1}(N)$ acts on $T_{-1}$ by a scalar, and also in $T_1$, so again we obtain $p^{-1}(N)\subset Z(\Gamma^{\prime})$, a contradiction. Hence $r\geqslant3$.

    In this case, assume first that $P^\prime\in X^\prime$ is smooth, i.e., $\dim T_{X^\prime,P^\prime}=3$, and again consider $z\in\mumu_r$ as a generator, which acts in $T_{X^\prime,P^\prime}$ as the diagonal matrix $\operatorname{diag}(\xi_r,\xi_r^a,\xi_r^{-a})$, where $\xi_r$ is a primitive $r$-th root of unity and $a$ is an integer coprime with $r$. Suppose first that $a\equiv\pm1$ mod $r$. In this situation, there is a decomposition
    $$T_{X^\prime,P^\prime}=T_1\oplus T_{-1},$$
    where $T_1,T_{-1}$ are the eigenspaces of $z$ with eigenvalues $\xi_r,\xi_r^{-1}$ respectively. Thus $\dim(T_1)=2$. Since $\mumu_r$ is normal, $\Gamma^\prime$ permutes its eigenspaces, preserving their dimensions. Hence both $T_1$ and $T_{-1}$ are $\Gamma^\prime$-invariant, and the preceding observation gives a contradiction. Then, we deduce that $a\not\equiv\pm1$ mod $r$, and $z$ has $3$ different eigenvalues.

    The permutation action of $\Gamma^\prime$ on the three $z$-eigenspaces factors through $\operatorname{Im}(\Phi)$, but since this is a $2$-group, $\Gamma^\prime$ will preserve one of these eigenspaces. Its eigenvalue will be $\xi_r^b$ with $\operatorname{gcd}(b,r)=1$. For any $g\in\Gamma^\prime$ we have that $gzg^{-1}=z^k$ for some integer $k$, and restricting this equality to the fixed eigenspace we get $\xi_r^b=\xi_r^{bk}$, so $k\equiv1$ mod $r$. In particular, this implies that the extension is central so $\Gamma^\prime$ preserves all the eigenspaces. Therefore, the action of $\Gamma^\prime$ on $T_{X^\prime,P^\prime}$ is diagonalizable, contradicting the fact that $\Gamma^\prime$ is not abelian.

    Now we assume that $P^\prime\in X^\prime$ is a singular point, so $\dim(T_{X^\prime,P^\prime})=4$. Again by the classification of terminal threefold singularities \cite{Reid:YPG}, we have the following two cases:
    \begin{enumerate}
        \item $r=4$ and $z$ acts diagonally as $\operatorname{diag}(\xi_r,\xi_r,\xi_r^2,\xi_r^3)$.
        \item $r\geqslant3$ and $z$ acts diagonally as $\operatorname{diag}(\xi_r^a,\xi_r^{-a},\xi_r,1)$, where $a$ is an integer coprime to $r$.
    \end{enumerate}
    If we are in the first case, the $\xi_4$-eigenspace has dimension $2$, so it is $\Gamma^\prime$-invariant. But then if $gzg^{-1}=z^k$ again we get $k\equiv1$ mod $4$ and the extension is central. In particular, $\Gamma^\prime$ preserves all the eigenspaces which are of dimensions $2,1,1$. The useful observation at the beginning gives a contradiction.
    
    Consequently, we are in the second case. Suppose that $a\equiv\pm1$ mod $r$. Again we observe that the whole eigenspace decomposition for $z$ is preserved under $\Gamma^\prime$, and we reach a contradiction in the same way. Thus, $z$ has $4$ different eigenvalues on $T_{X^\prime,P^\prime}$, and there is a decomposition into eigenspaces
    $$T_{X^\prime,P^\prime}=E_a\oplus E_{-a}\oplus E_1\oplus E_0.$$
    It is clear that $E_0$ is $\Gamma^\prime$-invariant, and then since the permutation action on the other three eigenspaces factors through $\operatorname{Im}(\Phi)$, it will fix a second eigenspace. The same computation as in the smooth case shows that $\mumu_r$ is central and therefore the $\Gamma^\prime$-action on $T_{X^\prime,P^\prime}$ is diagonalizable, a contradiction. This argument concludes the proof.
\end{proof}

\begin{lemma}
\label{lemma:A6-ext-ample}
Let $X$ be a Fano threefold with canonical Gorenstein singularities such that $\Aut(X)$ contains a subgroup isomorphic to $G$. Then $-K_X$ is very ample.
\end{lemma}

\begin{proof}
    Suppose by contradiction that $-K_X$ is not very ample. By Lemma \ref{lemma:A6-hyperelliptic} the linear system $\vert-K_X\vert$ gives a double cover $\phi:X\to Y\subset\PP^4$ where $Y$ is a $G$-invariant smooth quadric threefold. Such a quadric does not exist since $G$ does not admit an embedding into $\PGL_5(\CC)$.
\end{proof}

\begin{lemma}
    \label{lemma:A6-ext-index}
    Let $X$ be a Fano threefold with terminal singularities such that $\Aut(X)$ contains a subgroup isomorphic to $G$. Then its index is $\iota(X)=1$.
\end{lemma}

\begin{proof}
    We know that if $\iota(X)=4$ then $X\simeq\PP^3$ and if $\iota(X)=3$ then $X$ is a quadric hypersurface in $\PP^4$. These two cases cannot hold since $G$ does not embed into $\PGL_n(\CC)$ for $n=4,5$. Then, we assume $\iota(X)=2$ and we use the classification of terminal del Pezzo threefolds (see \cite{Shi89}). Consider $A$ an ample Cartier divisor such that $-K_X\sim2A$. In this context, $d=A^3$ is called the degree of $X$ and we know that $d\leqslant7$. If $d=1$, then $\operatorname{Bs}\vert A\vert$ consists of just one $G$-fixed point, which contradicts Theorem \ref{theorem:A6-non-fixed}. If $d=2$, $\vert A\vert$ gives a $G$-equivariant double cover $X\rightarrow\PP^3$, and this case is ruled out by Lemma \ref{lemma:A6-ext-projective-spaces}.

    For $d\geqslant3$, it is known that $A$ is very ample and gives a $G$-equivariant embedding $X\rightarrow\PP^{d+1}$. Let $V$ be a $(d+2)$-dimensional representation of $\widehat{G}=3.M_{10}$ inducing this projective action, and let $Z\simeq\mumu_3$ be the centre of $\widehat{G}$. Suppose that $3\leqslant d\leqslant5$. If $Z$ acts trivially on $V$, then $V$ descends to a faithful representation of $G$ of dimension at most $7$, which is impossible. Hence $Z$ acts by a non-trivial scalar character and \cite{lmfdb} gives $\dim(V)=6$, so $d=4$. In this case $X\subset\PP(V)$ is an intersection of two quadrics, so $I:=H^0(\PP(V),\mathscr{I}_X(2))$ is a $2$-dimensional $\widehat{G}$-representation and $Z$ acts trivially on $I$. However, a generator of $Z$ acts on $V$ by a primitive cubic root $\zeta$, so it acts on $\Sym^2(V^\vee)$ by $\zeta^{-2}\neq1$. This contradiction shows that $d\geqslant6$. 
    
    If $d=6$, then either $X\simeq\PP^1\times\PP^1\times\PP^1$ or $X$ is a $(1,1)$-divisor in $\PP^2\times\PP^2$. The former case cannot hold since automorphism group of $\PP^1\times\PP^1\times\PP^1$ does not contain $G$, and in the latter case since $X$ can have at most one singular point, by Theorem \ref{theorem:A6-non-fixed} $X$ has to be smooth. The second one is ruled out by Proposition \ref{proposition:W-actions}, so $d\neq6$. If $d=7$, then $X$ is the blow-up of $\PP^3$ at one point, and this blow-up must be $G$-equivariant, then this case is also ruled out. Since there is no possible case, we conclude the proof.
\end{proof}

\begin{lemma}[{\cite[Lemma 2.5.27]{CheltsovPinardinProkhorov}}]
    \label{lemma:A6-ext-quadrics}
    Let $X$ be a Fano threefold with canonical Gorenstein singularities such that $\Aut(X)$ contains a subgroup isomorphic to $G$. Then, its anticanonical image in $\PP^{g+1}$ is an intersection of quadrics.
\end{lemma}

\begin{proof}
    Since the space $H^0(X,\OOO_X(-K_X))$ is a faithful representation of the group $G$, it follows from \cite{lmfdb} that $g\geq7$, as $G$ does not have such representations of dimension less than $9$. Suppose now that $X$ is not an intersection of quadrics. Thus, it follows from \cite[Theorem 2.14]{HT} that the variety $Y\subset\PP^{g+1}$ cut out by all the quadrics containing $X$ is a $4$-dimensional irreducible variety of minimal degree. By the classification of such varieties we have that $Y$ is a rational normal scroll or a cone over a rational normal scroll. Then we can construct the $G$-equivariant diagram
    $$
    \xymatrix{
    X \ar@{^{(}->}[d] & \ar[l] \widetilde{X} \ar@{^{(}->}[d] \ar[r]^{\varphi} & \PP^1 \ar@{=}[d]\\
    Y  & \ar[l]_{\mu} \widetilde{Y} \ar[r]^{\pi} & \PP^1
    }
    $$
    where $\mu$ is an isomorphism if $Y$ is smooth and $\mu$ is the blow-up of the vertex if $Y$ is a cone. In this setting $\pi$ is a $\PP^3$-bundle and $\varphi$ is a fibration in cubic surfaces. Since $A$ acts trivially on $\PP^1$, it acts faithfully on the general fiber of $\varphi$, but this contradicts Theorem~\ref{theorem:A6-dP}. 
\end{proof}

Considering the previous lemma, we may assume now that $g=\mathrm{g}(X)\geqslant7$, $X$ is an intersection of quadrics in $\PP^{g+1}$ and the formula \cite[Corollary 5.5]{P:JAG:simple} applies.

\begin{lemma}
    \label{lemma:A6-ext-g-not-6-8}
    Suppose that $X$ is a canonical Gorenstein Fano threefold such that $\Aut(X)$ has a subgroup isomorphic to $G$, and let $A$ be the subgroup of $G$ with $A\simeq\mathfrak{A}_6$. Then $g\neq8$, and moreover, if $\vert-K_X\vert$ contains at most one $A$-invariant divisor then $g\notin\{10,12\}$.
\end{lemma}

\begin{proof}
    If $g=8$, $H^0(X,\OOO_X(-K_X))$ is a $10$-dimensional faithful representation of $G$, and by \cite[Corollary 5.5]{P:JAG:simple} $X$ is cut out by $15$ quadrics. The possible decompositions of this space into irreducible $G$-representations are
    \begin{enumerate}
        \item $H^0(X,\OOO_X(-K_X))=\mathbb{W}_{10}$;
        \item $H^0(X,\OOO_X(-K_X))=\mathbb{W}_1\oplus\mathbb{W}_{9}$;
    \end{enumerate}
        where each representation $\mathbb{W}_i$ can be any irreducible linear $G$-representation of dimension $i$. Here below we list the decompositions of $S^2H^0(X,\OOO_X(-K_X))$ in irreducible representations, indicating with multiplicities the number of blocks of each dimension (even if they are non-isomorphic). In each case, respectively, we get
    \begin{enumerate}
        \item $S^2H^0(X,\OOO_X(-K_X))=\mathbb{W}_1^{\oplus2}\oplus\mathbb{W}_9^{\oplus3}\oplus\mathbb{W}_{10}\oplus\mathbb{W}_{16}$ or \\
            $S^2H^0(X,\OOO_X(-K_X))=\mathbb{W}_1\oplus\mathbb{W}_9^{\oplus2}\oplus\mathbb{W}_{10}^{\oplus2}\oplus\mathbb{W}_{16}$;
        \item $S^2H^0(X,\OOO_X(-K_X))=\mathbb{W}_1^{\oplus2}\oplus\mathbb{W}_9^{\oplus3}\oplus\mathbb{W}_{10}\oplus\mathbb{W}_{16}$.
    \end{enumerate}
        In this case we need a $15$-dimensional $G$-invariant subspace of quadrics, but we see from the list of decompositions that this is not possible. 
    Finally, it is clear that if $\vert-K_X\vert$ contains at most one $A$-invariant divisor $g\notin\{10,12\}$, as in that case we see from \cite{lmfdb} that we would need at least two $1$-dimensional subrepresentations in $H^0(X,\OOO_X(-K_X))$. 
\end{proof}

\begin{lemma}
\label{lemma:A6-ext-threefolds-lc}
Suppose that $X$ is a terminal Fano threefold admitting a faithful action of the group $G$.
Let $D$ be a $G$-invariant effective $\mathbb{Q}$-divisor on $X$ such that $D\sim_{\mathbb{Q}}(-K_X)$. Then $(X,D)$ has log canonical singularities.
\end{lemma}

\begin{proof}
Suppose that the~singularities of the~log pair $(X,D)$ are not log canonical. Then, using \cite[Lemma~2.4.10]{Icosahedron} we see that there is a rational number $0<\lambda<1$ and an~effective $G$-invariant $\mathbb{Q}$-divisor $D^\prime$ on the~threefold $X$ such that $D^\prime\sim_{\mathbb{Q}} D$, the log pair $(X,\lambda D^\prime)$ has strictly log canonical singularities, and the~locus $\mathrm{Nklt}(X,\lambda D^\prime)$ consists of a~$G$-orbit of a minimal log canonical center of the~log pair $(X,\lambda D^\prime)$. Here, we do not assume that minimal log canonical center is irreducible.

It follows from \cite{Kawamata-1} that irreducible components of minimal log canonical centers of the~log pair $(X,\lambda D^\prime)$ are disjoint.
On the~other hand, it follows from the~Koll\'ar--Shokurov connectedness theorem that the~locus $\mathrm{Nklt}(X,\lambda D^\prime)$ is connected. Thus, we see that $\mathrm{Nklt}(X,\lambda D^\prime)$ consists of an irreducible $G$-invariant subvariety $Z$.

If $Z$ is a~point, we obtain a~contradiction with Theorem~\ref{theorem:A6-non-fixed}. Thus, either $Z$ is a~curve or a~surface. Then it follows from Kawamata's subadjunction theorem \cite{Kawamata-2} that $Z$ is normal and rational. Thus, if $Z$ is a~curve, then the action of $G$ on $Z$ is not faithful, by Lemma \ref{lemma:A6-ext-projective-spaces}. In this case, since $A$ is the unique non-trivial proper normal subgroup of $G$, the action of $A$ on $Z$ is trivial. Similarly, if $Z$ is a~surface, then $A$ acts trivially on $Z$ by Lemma~\ref{lemma:A6-ext-cr2}. Now, using Theorem~\ref{theorem:A6-non-fixed}, we obtain a~contradiction.
\end{proof}

\begin{lemma}
\label{lemma:A6-ext-threefolds-log-pair-mobile}
Suppose that $X$ is a~$G\mathbb{Q}$-Fano threefold. Let $\mathcal{M}$ be a~non-empty $G$-invariant mobile linear system on $X$ such that 
$$
\mathcal{M}\sim_{\mathbb{Q}}a(-K_X)
$$ 
for some positive rational number $a\leqslant 1$. Then $a=1$ and $(X,\mathcal{M})$ has canonical singularities.
\end{lemma}

\begin{proof}
    Assume that either $a<1$ or $(X,\mathcal{M})$ does not have canonical singularities. By \cite[Lemma 2.4.16]{CheltsovPinardinProkhorov} and the \hyperref[theorem:fibration]{Fibration Theorem} we get that there exists a $G$-equivariant birational map $X\dashrightarrow Y$ for some $G\QQ$-Fano threefold $Y$ with terminal Gorenstein singularities and $\iota(Y)\geqslant2$. This contradicts Lemma \ref{lemma:A6-ext-index}.
\end{proof}

\begin{corollary}
\label{corollary:A6-ext-positive-mobile}
Suppose that $X$ is a $G\mathbb{Q}$-Fano threefold and let $\mathcal{M}$ be a~non-empty $G$-invariant linear subsystem of $|-K_X|$ of positive dimension. Then $\mathcal{M}$ is mobile and $(X,\mathcal{M})$ has canonical singularities.
\end{corollary}

\begin{proof}
If $\mathcal{M}$ has a~fixed component, then its mobile part is $\mathbb{Q}$-linearly equivalent to $a(-K_X)$ for some positive rational number $a<1$, which contradicts Lemma~\ref{lemma:A6-ext-threefolds-log-pair-mobile}.
\end{proof}

\begin{corollary}
\label{corollary:A6-ext-K3-surfaces}
Suppose that $X$ is a $G\mathbb{Q}$-Fano threefold. Let $\mathcal{M}$ be a~non-empty $G$-invariant linear subsystem in $|-K_X|$ that has positive dimension, and let $S$ be a~general surface in $\mathcal{M}$. Then the~log pair $(X,S)$ has purely log terminal singularities, and $S$ is an~irreducible K3 surface with at most Du Val singularities.
\end{corollary}

\begin{proof}
    By Corollary \ref{corollary:A6-ext-positive-mobile} $\mathcal{M}$ is mobile and $(X,\mathcal{M})$ has canonical singularities. By Bertini's theorem the general member of $\mathcal{M}$ is irreducible, otherwise $\mathcal{M}$ would be composed with a mobile pencil, but this contradicts Lemma \ref{lemma:A6-ext-threefolds-log-pair-mobile}. Then by \cite[Theorem 4.8, Theorem 7.5]{KollarPairs} $(X,S)$ is purely log terminal and $S$ has klt singularities. By the adjunction formula and Kodaira vanishing $S$ is K3 with Du Val singularities. 
\end{proof}

\begin{lemma}
\label{lemma:A6-ext-invariant-surfaces}
Let $X$ be a~Fano threefold with terminal singularities that admits a faithful action of the group $G$. Suppose that $X$ is a~$G\mathbb{Q}$-Fano threefold, and $|-K_X|$ has a~$G$-invariant surface $S$. Then
\begin{itemize}
\item either $S$ is a~smooth K3 surface, and $G$ acts faithfully on $S$,
\item or $S$ is reducible, it has at least $10$ irreducible components, and $G$ acts transitively on them.
\end{itemize}
\end{lemma}

\begin{proof}
By Lemma~\ref{lemma:A6-ext-threefolds-lc}, the~log pair $(X,S)$ has log canonical singularities. This implies that $S$ is reduced.
Suppose now that $S$ is irreducible, and let us show that $S$ is a~smooth K3 surface. We first show that $S$ is normal. If $S$ is not normal, consider its normalization $\nu:\widehat{S}\to S$. Then adjunction gives $K_{\widehat{S}}+B\sim0$ for some non-zero effective integral $A$-invariant divisor such that $(\widehat{S},B)$ is log canonical. In particular, $B$ is reduced. Let $f:\widetilde{S}\to\widehat{S}$ be a minimal resolution, so $K_{\widetilde{S}}+f^{-1}_{\ast} B+F\sim f^{\ast}(K_{\widehat{S}}+B)$ where $F$ is $f$-exceptional. Since $K_{\widehat{S}}+B\sim0$, the divisor $F$ is integral, and by log canonicity the coefficients of $F$ are at most $1$. By minimality, $K_{\widetilde{S}}$ is $f$-nef, so $-F$ is also $f$-nef. By the negativity lemma, the last assertion implies that $F\geqslant0$, so $F+f_{\ast}^{-1}B$ is a non-zero $A$-invariant reduced divisor in $\vert-K_{\widetilde{S}}\vert$, which contradicts Lemma~\ref{Lemma:A6-ext-reduced-divisor}. We conclude that $S$ is normal.

If $S$ is smooth, it is clear that $S$ is a K3 surface. Thus, we may assume that $S$ is singular. By Theorem~\ref{theorem:A6-non-fixed} the action of $A$ on $S$ is non-trivial, then $G$ acts faithfully on $S$. It follows from Corollary~\ref{corollary:K3-a6} that $S$ has a~non-Du Val singular point $P$, but then its $A$-orbit contains at least $6$ points. This is not possible since $S$ can have at most two non-Du Val singular points \cite{Umezu}. Therefore, $S$ is a smooth K3 surface.

Now, we assume that $S$ is reducible. Suppose $G$ does not act transitively on the set of irreducible components of $S$ and consider $S^\prime$ as the $G$-orbit of an irreducible component of the surface $S$. Then $S^\prime$ is $G$-invariant and
$$
S^\prime\sim_{\mathbb{Q}} a(-K_X)
$$
for some $a\in\mathbb{Q}_{>0}$ such that $a<1$. In this case, the log pair $(X,\frac{1}{a}S^\prime)$ is not log canonical, which contradicts Lemma \ref{lemma:A6-ext-threefolds-lc}. Thus, $a=1$ and we conclude that $S^\prime=S$, i.e., $G$ acts transitively on the irreducible components of $S$.

Finally, we claim that $S$ cannot have $2$ components. Suppose $S=S^\prime+S^{\prime\prime}$ with $S^\prime,S^{\prime\prime}$ irreducible and $A$-invariant. Let $\nu:\widehat{S}^\prime\to S^\prime$ be the normalization of $S^\prime$, and consider the divisor $B_{\widehat{S}^\prime}$ as the different divisor of the log pair $(X,S^\prime+S^{\prime\prime})$, i.e., $B_{\widehat{S}^\prime}$ is the effective Weil divisor on $\widehat{S}^\prime$ such that
$$
K_{\widehat{S}^\prime}+B_{\widehat{S}^\prime}\sim \nu^\ast((K_X+S^\prime+S^{\prime\prime})\vert_{S^\prime}).
$$
Thus, $B_{\widehat{S}^\prime}\in\vert-K_{\widehat{S}^\prime}\vert$ is $A$-invariant, and by \cite{Kawakita} the pair $(\widehat{S}^\prime,B_{\widehat{S}^\prime})$ has log canonical singularities. In particular, $B_{\widehat{S}^\prime}$ is reduced. Replacing the log pair $(\widehat{S}^\prime,B_{\widehat{S}^\prime})$ by its minimal resolution $(\widetilde{S}^\prime,B_{\widetilde{S}^\prime})$, we find a smooth surface $\widetilde{S}^\prime$ with a faithful action of $A$ which contains a non-zero $A$-invariant reduced divisor $B_{\widetilde{S}^\prime}\in\vert -K_{\widetilde{S}^\prime}\vert$, which contradicts Lemma~\ref{Lemma:A6-ext-reduced-divisor}. Thus, $S$ cannot have $2$ irreducible components and then if $S$ is reducible it has at least $10$ components, as the only proper subgroup of $G$ with index less than $10$ is $A$.
\end{proof}

\subsection{Proof of Fibration Theorem}
\label{subsection:fibrations}

Here below we prove the following reduction theorem.

\begin{fibrationprop*}
    \phantomsection
    Let $X$ be a rationally connected threefold with terminal singularities that admits a faithful action of one of the groups $\PGL_2(\FF_7),\PGL_2(\FF_9),M_{10}$. Then $X$ is $G$-birational  neither to a $G$-conic bundle nor to a $G$-del Pezzo fibration.
\end{fibrationprop*}

\begin{proof}
Suppose that the assertion is not true. Then there exists a $G$-birational map $X\dasharrow Y$ such that $Y$ is a normal threefold with terminal $G\mathbb{Q}$-factorial singularities, and there exists a $G$-equivariant surjective morphism $\pi\colon Y\to Z$ with connected fibers such that $Z$ is normal and one of the following cases holds:
\begin{itemize}
\item $Z$ is a rational surface, the general fiber of $\pi$ is a rational curve, $\mathrm{rk}\,\mathrm{Pic}(Y)^G=\mathrm{rk}\,\mathrm{Pic}(Z)^G+1$;
\item $Z=\mathbb{P}^1$ and the general fiber of $\pi$ is a smooth del Pezzo surface, $\mathrm{rk}\,\mathrm{Pic}(Y)^G=2$.
\end{itemize}
Let $A$ be the subgroup of $G$ isomorphic to $\PSL_2(\FF_7)$ or $\mathfrak{A}_6$, as appropiate. 
Then $A$ acts trivially on $Z$, because the $G$-action on $Z$ cannot be faithful by \cite{IgorVasya}. If $Z$ is a surface, since $A$ does not act faithfully on the general fiber of $\pi$, we see that $A$ acts trivially on $Y$, which is absurd. Thus, we see that $Z=\mathbb{P}^1$.

Let $F$ be a general fiber of $\pi$. Then it follows from \cite[Lemma 2.5.5]{CheltsovPinardinProkhorov} that either $F\simeq\mathbb{P}^2$ or $F$ is isomorphic to a smooth del Pezzo surface of degree $2$ described as a double cover of $\PP^2$ branched at the plane Klein quartic curve. If $G$ acts trivially on $Z$, then $F$ is $G$-invariant, and $G$ acts faithfully on $F$, which is impossible by \cite{IgorVasya}. Hence, we conclude that $G$ acts on $Z$ as $G/A\cong\mumu_2$. In particular, $G$ fixes two points in $Z$. 

Let $P$ be one of the $G$-fixed points in $Z$, and let $S$ be the scheme-theoretic fiber of $\pi$ over this point. Choose a positive rational number $\mu\leqslant 1$ such that the singularities of the log pair $(Y,\mu S)$ are strictly log canonical. Let $C$ be an irreducible subvariety in $Y$ that is a minimal center of log canonical singularities of this log pair, and let $T$ be its $G$-orbit in $Y$. Then
$$
T\subset\mathrm{Supp}(S),
$$ 
and every irreducible component of $T$ is also a minimal center of log canonical singularities of the log pair $(Y,\mu S)$, which implies that these irreducible components are disjoint, because intersection of two log canonical centers is a log canonical center by \cite[Proposition 1.5]{Kawamata-1}. Now, arguing as in the proof of \cite[Lemma 2.8]{cheltsov2011exceptional}, we can construct an effective $G$-invariant $\mathbb{Q}$-divisor $D$ on $Y$ such that 
$$
D\sim_{\mathbb{Q}}\varepsilon(-K_{Y})+S
$$ 
for a very small rational number $\varepsilon>0$, the singularities of the log pair $(Y,\lambda D)$ are strictly log canonical for some positive rational number $\lambda>\mu$, $C$ is a minimal center of log canonical singularities of this pair, and 
$$
\mathrm{Nklt}\big(Y,\lambda D\big)=T.
$$
Since $-(K_Y+\lambda D)$ is $\pi$-ample, it follows from Koll\'ar--Shokurov connectedness \cite[Theorem 7.4]{KollarPairs} that the locus $\mathrm{Nklt}(Y,\lambda D)$ is connected. Hence, we see that $T=C$, so $C$ is $G$-invariant. In particular, if $C$ is a point, we obtain a contradiction with Lemma~\ref{lemma:A6-ext-fixed} (and similar statement works for $A=\PSL_2(\FF_7)$).

Thus, we see that $C$ is either a curve or a surface and $G$ acts faithfully on $C$ by Lemma~\ref{lemma:A6-ext-fixed} and \cite[Corollary 2.5.13]{CheltsovPinardinProkhorov}. On the other hand, it follows from \cite{Kawamata-2} that $C$ is normal. Fix any sufficiently small ample $\mathbb{Q}$-divisor $L$ on $Y$. Then it follows from Kawamata's subadjunction theorem \cite[Theorem 1]{Kawamata-2} that there exists an effective $\mathbb{Q}$-divisor $\Delta_C$ on $C$ such that $(C,\Delta_C)$ has Kawamata log terminal singularities and
$$
(K_{Y}+\lambda D+L)\big\vert_{C}\sim_{\mathbb{Q}} K_C+\Delta_C.
$$ 
Since $\varepsilon$ is very small and $L$ can be chosen to be sufficiently small, we see that the divisor $-(K_{Y}+\lambda D+L)$ is $\pi$-ample. Thus, if $C$ is a curve, we have $\mathrm{deg}(K_C)<0$, so $C$ is rational, which is a contradiction. Similarly, if $C$ is a surface, we see that $(C,\Delta_C)$ is a log del Pezzo surface with Kawamata log terminal singularities. This also implies that $C$ is rational, again impossible by \cite{IgorVasya}. 
\end{proof}

\section{Proof of the Technical Theorem: Gorenstein case}
\label{section:A6-ext-gorenstein}

Let $X$ be a Fano threefold with terminal Gorenstein singularities such that $\mathrm{Aut}(X)$ contains some finite subgroup $G$. Suppose in addition that $X$ is rational, so from \hyperref[theorem:fibration]{Fibration Theorem} we may assume that $X$ is a $G\QQ$-Fano $3$-fold. In this section, we prove the following.

\begin{theorem}
\label{theorem:A6-ext-Gorenstein}
One has $G\not\simeq M_{10}$.
\end{theorem}

Let us assume by contradiction that $G\simeq M_{10}$ and denote by $A\subset G$ the unique subgroup isomorphic to $\mathfrak{A}_6$. By Lemma \ref{lemma:A6-ext-ample} we know that the anticanonical linear system $\vert-K_X\vert$ gives a $G$-equivariant embedding $X\hookrightarrow\PP^{g+1}$ where
$$g=\mathrm{g}(X)=\frac{(-K_X)^3+2}{2}\geqslant7,$$
and moreover, from Lemma \ref{lemma:A6-ext-quadrics} $X$ is given as an intersection of quadrics in $\PP^{g+1}$. Then we have from \cite[Corollary 5.5]{P:JAG:simple} that
$$h^0(\PP^{g+1},\mathscr{I}_X(2))=\frac{(g-2)(g-3)}{2}.$$
In the following we identify $X$ with its anticanonical image in $\PP^{g+1}$. By \cite{Namikawa,SmoothingPic} we know that $X$ admits a $\QQ$-Gorenstein smoothing to a smooth Fano $3$-fold $V$ such that
$$(-K_V)^3=(-K_X)^3=2g-2\geq10$$
and we also have that $\rho(V)=\rho(X)$ and $\iota(V)=\iota(X)=1$ by Lemma \ref{lemma:A6-ext-index}. During the rest of the section $V$ will always denote a general fiber of such a smoothing. Furthermore, we know that 
\begin{equation}
\label{equation:sing}
\vert\Sing(X)\vert\leqslant20+h^{1,2}(V)-\rho(V)\leqslant19+h^{1,2}(V).
\end{equation}

\begin{lemma}
    \label{lemma:A6-ext-pic}
    One has $\rho(X)=1$.
\end{lemma}

\begin{proof}
    We assume that $\rho(X)>1$ and follow the classification \cite{P:GFano2}. Considering $V$ a $\QQ$-Gorenstein smoothing of $X$, by \cite{P:GFano2} we see that $\rho(X)=\rho(V)\leqslant4$, so in particular $A$ acts trivially on $\Pic(X)$. Consequently, the action on the Picard lattice factors through $G/A\simeq\mumu_2$, and an involution $\tau\in G/A$ acts non-trivially on $\Pic(X)$, otherwise we would have $\rho(X)^G=\rho(X)>1$. Note that if we are in cases (1.2.5),(1.2.7) or (1.2.8) of \cite[Theorem 1.2]{P:GFano2}, $V$ admits a morphism $p:V\to\PP^1$ with connected fibers. Let $F$ be a general fiber of $p$, $f=[F]$ its class on $\Pic(V)$ and set $f^\prime=\tau(f),h=[-K_V]$. Since $F$ is a fiber, $f^2\cdot u=0$ for every divisor class $u$. Preservation of the intersection form gives
    $$(f^\prime)^2\cdot u=f^2\cdot\tau^{-1}(u)=0.$$
    As $f+f^\prime$ is $G$-invariant,  we get that $f+f^\prime=ah$ for some $a>0$. But then
    $$0=(f+f^\prime)^3=a^3h^3>0,$$
    a contradiction.

    If we are in case (1.2.6) of \cite[Theorem 1.2]{P:GFano2}, the nef cone $\operatorname{Nef}(V)$ is generated by three classes, so the involution $\tau$ has to preserve one of these generators. This contradicts the assumption $\rho(X)^G=1$.

    The remaining cases have $\rho(V)=2$. If we are in case (1.2.4), \cite[Theorem 6.5]{P:GFano2} implies that $X$ is smooth, and this case is not possible due to Proposition \ref{proposition:W-actions}. Assume we are in case (1.2.3). Again by \cite[Theorem 6.5]{P:GFano2} there is an $A$-equivariant birational morphism $X\to Q$ where $Q$ is a quadric in $\PP^4$. There are two possibilities:
    \begin{enumerate}
        \item either $X$ is smooth, in which case $X$ is the blow-up of $Q$ along a twisted quartic curve, and $A$ fixes this curve pointwise;
        \item or $X$ is singular, and then by \cite[Theorem 6.5]{P:GFano2} $Q$ is a quadratic cone whose vertex is $A$-fixed.
    \end{enumerate}
    In any case, $Q$ has terminal singularities, so both situations contradict Theorem \ref{theorem:A6-non-fixed}.

    Now suppose that we are in case (1.2.2). Then by \cite[Theorem 6.5]{P:GFano2} $X$ itself is an intersection of three divisors of degree $(1,1)$ in $\PP^3\times\PP^3$. The first coordinate projection is $A$-equivariant, and arguing as in \cite[2.3]{P:GFano2} one sees that, away from finitely many points in $\PP^3$, it is the blow-up of $\PP^3$ along a sextic curve. Now by Lemma \ref{lemma:A6-invariant-curves} the curve $C$ is a $G^\prime$-orbit of six lines in $\PP^3$, but in that case the blow-up $X$ along these lines is such that $-K_X$ is not ample, which is a contradiction. 
    
    Thus, assume we are in case (1.2.1) i.e., either $X$ corresponds to a divisor of degree $(2,2)$ in $\PP^2\times\PP^2$ or there exists an $A$-equivariant ramified double cover $X\to W$ with $W$ a smooth divisor of degree $(1,1)$ in $\PP^2\times\PP^2$. In both cases, the action of $G^\prime$ lifts to $\PP^2\times\PP^2$ (it preserves both coordinate projections), and this action is given by two $3$-dimensional representations of the Valentiner group $3.\mathfrak{A}_6$. It can be checked that the action is not diagonal, since in this case there are no invariants of degree $(2,2)$. Now, if the action is twisted diagonal, there is a unique $(2,2)$-degree invariant which is a multiple of the $(1,1)$-degree invariant divisor. Since there is no possibility, we conclude that $\rho(X)=1$.
\end{proof}

By the classification of smooth Fano $3$-folds, with $\rho(V)=\iota(V)=1$, we have that $g\in\{7,8,9,10,12\}$.

\begin{lemma}
\label{lemma:A6-ext-Gorenstein-G'-inv}
    The linear system $\vert-K_X\vert$ has at most one $A$-invariant surface.
\end{lemma}

\begin{proof}
    Assume that there exist two different $A$-invariant surfaces in $\vert-K_X\vert$. Then, there exists two distinct $G$-invariant surfaces $S,S^\prime\in\vert-K_X\vert$, and they generate a pencil $\mathcal{M}\subset\vert-K_X\vert$ such that every member is $A$-invariant (but not necessarily $G$-invariant).
    
    First, we prove that every $G$-invariant member of $\mathcal{M}$ is an irreducible smooth K3 surface. Suppose this is not true for $S$. By Lemma \ref{lemma:A6-ext-invariant-surfaces} we have that $S$ has at least $10$ irreducible components and $G$ acts transitively on them. Since $\deg(S)=2g-2\in\{12,14,16,18,22\}$, we see that the only possible case is $g=7$ and $S=S_1+\ldots+S_{12}$. In such a case, each component $S_i$ would be a plane, but this contradicts \cite{Prokhorov-G-Fanos}. Thus, every $G$-invariant member of $\mathcal{M}$ is indeed an irreducible smooth K3 surface, and in particular $X$ is smooth along $S$.

    Observe that Theorem \ref{theorem:A6-non-fixed} implies that $G$ acts faithfully on $S$ and $S^\prime$, and by Lemma \ref{lemma:K3-a6} we have that $\Pic(S)^{A}=\mathbb{Z}[L]$ for some ample divisor $L\in\Pic(S)$ such that $L^2$ is an even integer. Consider $Z=S^\prime\vert_S$, which is an $A$-invariant curve and then $Z\sim nL$ for some $n\geqslant1$. We claim that $Z$ is irreducible. Note that
    $$
    n^2L^2=Z^2=2g-2\in\{12,14,16,18,22\}.
    $$
    and then if $n\neq1$ we should have $g=9,L^2=4$ and $n=2$. Indeed, if $n=3$ we get $L^2=2$ and then a $G$-equivariant double cover $\phi_{\vert L\vert}:S\to\PP^2$, which contradicts Lemma~\ref{lemma:A6-ext-projective-spaces}.
    
    First, we note that $Z$ is reduced and $A$-irreducible. If this is not true, we can write $Z=Z_1+Z_2$ where $Z_1,Z_2$ are $A$-irreducible and $Z_1\sim Z_2\sim L$. If $Z_1\neq Z_2$, the intersection $Z_1\cap Z_2$ consists of an $A$-orbit of at most $4$ points, which is not possible since $A$ does not have a proper subgroup of index $\leqslant4$. If $Z_1=Z_2$, then $Z_1^2=L^2=4$, and by the adjunction formula
    $$\mathrm{p}_a(Z_1)=1+\frac{1}{2}(Z_1^2+Z_1\cdot K_S)=3,$$
    but this is impossible by Lemma \ref{lemma:A6-ext-curves}. We conclude that $Z$ is reduced and $A$-irreducible. Moreover, given that $Z$ is $A$-irreducible and $G$-invariant, it is also $G$-irreducible.
    
    Assume now that $Z$ is reducible, so we can write $Z=C_1+\ldots+C_m$ where each $C_i$ is an irreducible curve and $m>1$. Note that $m$ divides $\deg(Z)=2g-2$, so the only possibility is that $g=7$, $m=12$ and each curve $C_i$ is a line. We obtain that
    $$
    1=\deg(C_1)=-K_X\vert_S\cdot C_1=Z\cdot C_1=-2+\sum_{i=2}^{m}C_1\cdot C_i,
    $$
    and then it follows that $\vert C_1\cap(C_2\cup\ldots\cup C_{m})\vert\leqslant3$. Note that the stabilizer $G_1$ of $C_1$ is either $\mathfrak{A}_5$, and these groups do not fix points in $S$ (they do not have $2$-dimensional faithful representations). Then $G_1$ leaves the set $C_1\cap(C_2\cup\ldots\cup C_{m})$ invariant in $C_1\simeq\PP^1$. We observe that this is not possible, as these groups do not have cyclic subgroups of index $\leqslant3$. It follows that $Z$ is irreducible.
    By the adjunction formula we have that $\mathrm{p}_a(Z)=g(X)$ and $G$ acts faithfully on $Z$, so Corollary \ref{corollary:A6-ext-singular-curves} implies that $Z$ is a smooth curve. However, Lemma \ref{lemma:A6-ext-curves} excludes the possibility of $g(Z)=g(X)\geqslant16$, a contradiction.
\end{proof}

From Lemma \ref{lemma:A6-ext-Gorenstein-G'-inv} and Lemma \ref{lemma:A6-ext-g-not-6-8} we conclude the following.

\begin{corollary}
    \label{corollary:A6-ext-g-7-9}
    One has $g\in\{7,9\}$. In particular, if $g=9$ the linear system $\vert-K_X\vert$ contains a unique $G$-invariant member.
\end{corollary}

\subsection{Fano $3$-folds of genus $g=7$}

During this section we assume that $g=7$. First, note that if $X$ is smooth \cite[Corollary 4.3.5]{KPS:Hilb} states that there is an embedding $\Aut(X)\hookrightarrow\Aut(C)$ where $C$ is a smooth irreducible curve of genus $7$, but this contradicts Lemma~\ref{lemma:A6-ext-curves}. Moreover, we can see that $X$ is not $\QQ$-factorial. Indeed, if $X$ were $\QQ$-factorial, cite{SingPoints} implies that $\vert\operatorname{Sing}(X)\vert\leqslant7$, but then this means that each singular point is $A$-fixed, and this situation cannot happen by Theorem \ref{theorem:A6-non-fixed}. Also notice that, since $G$ does not have faithful representations of dimension $\leqslant8$, the space of sections $H^0(X,\OOO_X(-K_X))$ is an irreducible $G$-representation. In particular, there are no $G$-invariant members in $\vert-K_X\vert$.

\begin{lemma}
    \label{lemma:lines-g7}
    Let $V$ be a smooth Fano $3$-fold with $\Pic(V)=\mathbb{Z}[-K_V]$ and genus $g=7$. If $D$ denotes the divisor spanned by lines in $V$, then $D\sim -7K_V$.
\end{lemma}

\begin{proof}
    It is known from \cite[Theorem 3.4]{Isk2} that $V$ contains lines, they generate a surface $D$, and if $D\sim a(-K_V)$ then $k=a+1$ is the number of lines that intersect a general line in $V$. Let $\ell$ be a general line in $V$. In the following, we compute $k$. From \cite[Theorem 4.3.3]{IP99} we know that there exists a commutative diagram
    $$
\xymatrix{
 & & \widetilde{V} \ar[lld]_{\alpha} \ar@{-->}[rrrr]^\chi \ar[rrd]^\beta & & & & \widetilde{V}^+ \ar[lld]_\gamma \ar[rrd]^\delta & & \\
V \ar@{-->}[rrrr] & & & & V^\prime & & & & \ar@{-->}[llll] \PP^1
}
$$
where $\alpha$ is the blow-up along $\ell$, $\beta$ and $\gamma$ are small birational morphisms, $\chi$ is a composition of flops, $V^\prime$ is a Fano $3$-fold with terminal Gorenstein singularities of genus $6$ and $\delta$ is a del Pezzo fibration of degree $5$. Denote by $E$ the exceptional divisor of $\alpha$, and by $E^+:=\chi_*(E)$ its strict transform on $\widetilde{V}^+$. Then, it can be seen that $k=E^3-(E^+)^3$ (see \cite[Remark 4.1.3]{IP99}). 
In order to compute this number, we note that if $\Pic(\PP^1)=\mathbb{Z}[M]$ one has the relation $\delta^*(M)\sim-K_{\widetilde{V}^+}-E^+$, and then using \cite[Lemma 4.1.2]{IP99} it follows that
\begin{align*}
    0=\delta^*(M)^3 = (-K_{\widetilde{V}^+})^3-3(-K_{\widetilde{V}^+})^2\cdot E^+ +3(-K_{\widetilde{V}^+})\cdot (E^{+})^2-(E^{+})^3=-7-(E^{+})^3,
\end{align*}
and then $k=E^3-(E^{+})^3=8$.
\end{proof}

\begin{lemma}
    \label{lemma:divisor-of-lines-g7}
    Let $D$ be the divisor spanned by all lines in $X$. Then $D\sim a(-K_X)$ for $a~\in~\{1,\ldots,7\}$.
\end{lemma}

\begin{proof}
    It is clear that $D\sim a(-K_X)$ for some positive integer $a$, as $D$ is $G$-invariant. The fact that $a\leqslant7$ comes from the same computation as in \cite[Lemma 2.6.9]{CheltsovPinardinProkhorov} using Lemma \ref{lemma:lines-g7}.
\end{proof}

\begin{lemma}
    \label{lemma:A6-ext-gorenstein-g7-singular-points}
    The singular locus $\operatorname{Sing}(X)$ consists of a $G$-orbit of $20$ points.
\end{lemma}

\begin{proof}
If $V$ is a $\QQ$-Gorenstein smoothing of $X$ we have that $h^{1,2}(V)=7$, and then \eqref{equation:sing} gives us $\vert\Sing(X)\vert\leqslant26$. According to Theorem \ref{theorem:A6-non-fixed} the locus $\operatorname{Sing}(X)$ can only consist of $G$-orbits of length $10,12$ or $20$. By Lemma \ref{lemma:A6-ext-fixed} it is not possible to have a $G$-orbit of length $10$. Assume now that the $G$-orbit of $P$ has length $12$, so its stabilizer is $G_P\simeq\mathfrak{A}_5$. In this case one can compute that the space of sections $H^0(X,-K_X)$ splits as a sum of one $4$-dimensional and one $5$-dimensional representations of $G_P$. This means that $G_P$ does not fix any point in $X$, a contradiction. The conclusion follows.
\end{proof}

\begin{lemma}
    \label{lemma:A6-ext-g-7-div}
    If $E\subset X$ is an $A$-invariant irreducible surface, then $\deg(E)$ is divisible by $6$ and $\deg(E)\geqslant12$.
\end{lemma}

\begin{proof}
   Consider $E_1$ an irreducible $A$-invariant surface. If $S=E_1+E_2$ is its $G$-orbit, then $S\sim a(-K_X)$ for some positive integer $a$, and $2\deg(E_1)=12a$. Since there are no $G$-invariant members in $\vert-K_X\vert$, we have that $a\geqslant2$, so in particular $\deg(E_1)\geqslant12$ and $\deg(E_1)$ is divisible by $6$.
\end{proof}

It is a consequence of \cite{Prokhorov-G-Fanos} that the $3$-fold $X$ does not contain any planes, i.e., there is no irreducible surface $E\subset X$ such that $(-K_X)^2\cdot E=1$. Now we can run the MMP inside the category of almost Fano threefolds, as it is explained in \cite[3.3]{Prokhorov-g-12}. More specifically, we consider $f:\widetilde{X}\to X$ a $\QQ$-factorialization of $X$, so $\widetilde{X}$ is an almost Fano threefold with no planes such that $\rho(\widetilde{X})=\mathrm{r}(X)>1$ and $-K_{\widetilde{X}}\sim f^\ast(-K_X)$, and we obtain the diagram
\begin{equation}
\label{equation:MMP}
\xymatrix{
&\widetilde{X}\ar[ld]_{f}\ar[rr]^{h_1}&&\widetilde{X}_1\ar[rr]^{h_2}&&\cdots\ar[rr]^{h_{n}}&&\widetilde{X}_n\ar[rd]^{\eta}&\\
X&&&&&&&& Z}
\end{equation}
where each $\widetilde{X}_i$ is an almost Fano threefold, each $h_i$ is an extremal birational contraction of a~surface to a curve (note that $n=0$ is possible), and $\eta$ is an extremal non-divisorial contraction such that $\mathrm{dim}(Z)\in\{0,1,2\}$. We set $h=h_n\circ\ldots\circ h_1$ and $h=\operatorname{id}$ if $n=0$.

\begin{lemma}
    \label{lemma:A6-ext-g-7-faithful}
    The group $G$ acts faithfully on $\operatorname{Cl}(X)$.
\end{lemma}

\begin{proof}
    Suppose by contradiction that $A$ acts trivially on $\operatorname{Cl}(X)$, so the whole diagram \ref{equation:MMP} is $A$-equivariant. Assume first that $\eta$ is a del Pezzo fibration. Consider $F$ as a general fiber of the composition $\eta\circ h:\widetilde{X}\to Z=\PP^1$, which is a smooth del Pezzo surface. Since $A$ acts trivially on $Z$, we necessarily have that $F$ is $A$-invariant but $(-K_F)^2=\deg(F)\geqslant12$, a contradiction.

Now, we will prove that $n=0$, or equivalently, that $\widetilde{X}$ does not admit any birational contraction. Suppose $n>0$, so each $h_i$ is the $A$-equivariant blow-up of an irreducible curve $C_i$. Note that $A$ acts faithfully on each $C_i$ by Theorem~\ref{theorem:A6-non-fixed}. If $C_1$ is singular, by Lemma \ref{lemma:A6-singular-curves} we have that $\mathrm{p}_a(C_1)\geqslant46$. Then, using \cite[Section 4.2]{CheltsovPinardinProkhorov}, we get that
$$
(-K_{\widetilde{X}_1})^3=(-K_{\widetilde{X}})^3+2(-K_{\widetilde{X}})^2\cdot\widetilde{E}_1+2\mathrm{p}_a(C_1)-2\geq76,
$$
    which is not possible since $(-K_{\widetilde{X}_1})^3\leqslant64$. Then $C_1$ has to be smooth, and similarly one proves that every $C_i$ is smooth.
    As every $C_i$ has to be irreducible, by Lemma~\ref{lemma:A6-genus} we have that $g(C_i)\geqslant10$ for every $i\in\{1,\ldots,n\}$ and then using Lemma \ref{lemma:A6-ext-g-7-div} and \cite[Section 4.2]{CheltsovPinardinProkhorov} we obtain the inequality
    \begin{align*}
        (-K_{\widetilde{X}_n})^3\geqslant(-K_{\widetilde{X}})^3 + \sum_{i=1}^n(2(-K_{\widetilde{X}})^2\cdot\widetilde{E}_i+2g(C_i)-2) \geqslant 12+42n,
    \end{align*}
    which implies that $n=1$, as $(-K_{\widetilde{X}_n})^3\leqslant64$. Moreover, the inequality above is an equality for $n=1$ and gives us that
    \begin{center}
        $\deg(\widetilde{E}_1)=(-K_{\widetilde{X}})^2\cdot\widetilde{E}_1=12\;\;$, $\;\; g(C_1)=10\;\;$ and $\;\; (-K_{\widetilde{X}_1})^3=54$.
    \end{center}
    If we assume that $\rho(\widetilde{X}_1)=1$, then we have that $\widetilde{X}_1$ is a smooth (by construction $\widetilde{X}_1$ is $\QQ$-factorial and then we use \cite[Corollary 4.2.2]{Prokhorov-g-12}) quadric $3$-fold in $\PP^4$ that contains a curve $C_1$ of genus $g(C_1)=10$. Moreover, we see that
    $$12=(-K_X)^3=54+2K_{\widetilde{X}_1}\cdot C_1+2g(C_1)-2=72+2K_{\widetilde{X}_1}\cdot C_1$$
    gives $\deg(C_1)=10$ in $\PP^4$.
    
    Notice that the $A$-action on $\widetilde{X}_1$ extends to a faithful action in $\PP^4$, so it is given by a $5$-dimensional irreducible faithful representation of $A$. In particular, this implies that $C_1$ is not contained in any hyperplane since otherwise its linear span would be an $A$-invariant subspace. Then, \cite[Theorem 2.20]{cheltsov2014five} gives us $g(C_1)\leqslant9$, a contradiction.
    
    Then, we are led to the case $\rho(\widetilde{X}_1)\geqslant2$ and $\eta:\widetilde{X}_1\to Z\simeq\PP^2$ a conic bundle. Let $L$ be a general line in $Z$ and let $H=(\eta\circ h)^\ast(L)$ be smooth and irreducible in $\widetilde{X}$. Then the restriction $\phi:=(\eta\circ h)\vert_H:H\to L$ is a conic bundle. Notice that $(\eta\circ h)(\widetilde{E}_1)$ has to be an $A$-invariant curve, since $A$ does not fix any point on $Z$. Now, by the adjunction formula one can compute that
    $$(-K_{\widetilde{X}})^2 \cdot H=4+K_H^2\leqslant12,$$
    On the other hand, if $\Delta$ denotes the union of the discriminant curve of the conic bundle $\eta$ and $(\eta\circ h)(\widetilde{E}_1)$, we have that $\deg(\Delta)=8-K_H^2$. Since $\Delta$ is an $A$-invariant curve, from Lemma \ref{lemma:A6-invariant-curves-P2} we obtain that $\deg(\Delta)=6$. Now, as $\Delta$ contains $(\eta\circ h)(\widetilde{E}_1)$ and the discriminant locus of $\eta$ and both are $A$-invariant, we should have that the discriminant locus of $\eta$ is empty. In other words, $\widetilde{X}_1$ is smooth and $\eta:\widetilde{X}_1\to\PP^2$ is a $\PP^1$-bundle. This means that $\widetilde{X}_1=\PP_{\PP^2}(\mathscr{E})$ where $\mathscr{E}$ is a vector bundle of rank $2$ on $\PP^2$ and we write $c_1(\mathscr{E})=a\ell, c_2(\mathscr{E})=b\ell^2$ for $a,b\in\mathbb{Z}$ and $\ell$ the class of a line in $\PP^2$. Then we have
    $$-K_{\widetilde{X}_1}=2\xi+(3-a)\eta^\ast\ell$$
    with $\xi=\OOO_{\PP(\mathscr{E})}(1)$ the tautological bundle. We can compute that
    $$(-K_{\widetilde{X}_1})^3=2c_1(\mathscr{E})^2-8c_2(\mathscr{E})+6K_{\PP^2}^2=2a^2-8b+54=54,$$
    so we obtain that $b=a^2/4$. Then $a=2m$ and $b=m^2$ for some $m\in\mathbb{Z}$, and the bundle $\mathscr{E}^\prime=\mathscr{E}\otimes\OOO_{\PP^2}(-m)$ is such that $\PP(\mathscr{E}^\prime)\simeq\PP(\mathscr{E})$ and $c_1(\mathscr{E}^\prime)=c_2(\mathscr{E}^\prime)=0$. Then, \cite[Proposition 2.10]{almost-fano-bundles} implies that $\mathscr{E}^\prime=\OOO\oplus\OOO$ is the trivial bundle and $\widetilde{X}_1\simeq\PP^1\times\PP^2$. Now, since the projection $\widetilde{X}_1\to\PP^2$ is $A$-equivariant, $A$ preserves the cone of curves of $\widetilde{X}_1$ and the $\PP^2$-bundle $\widetilde{X}_1\to\PP^1$ will be $A$-equivariant, a contradiction.

    Hence, we see that $n=0$ and $\eta:\widetilde{X}\to Z$ is a conic bundle. Moreover, it follows from Theorem \ref{theorem:A6-dP} that $Z\simeq\PP^2$. By \cite[Section 2.4]{CheltsovPinardinProkhorov} there exists a diagram
    \begin{equation}
    \label{equation:conic-bundles}
    \xymatrix{
    \widetilde{X}\ar[d]_{h}\ar[rrd]_{f}\ar@{-->}[rrrr]^{\chi}&&&&\widetilde{X}^\prime\ar[d]^{h^\prime}\ar[lld]^{f^\prime}\\
    \mathbb{P}^2&&X&&\mathbb{P}^2}
    \end{equation}
    where $f^\prime$ is a $\mathbb{Q}$-factorialization of the~Fano threefold $X$, $\chi$ is a~non-biregular composition of flops, and~$h^\prime$ is an extremal contraction of $\widetilde{X}^\prime$ which is a conic bundle. Consider an element $\tau\in G\setminus A$ with $G/A\simeq\langle\tau\rangle$. Then set $\widetilde{H}$ as a general member of $\vert h^\ast\OOO_{\PP^2}(1)\vert$, $H=f(\widetilde{H})$ and $H^\prime=\tau(H)$.
    
    Since the discriminant curve is $A$-invariant, $H+H^\prime\sim\lambda(-K_X)$ and
    $$2(12-d)=\lambda(-K_X)^3=12\lambda,$$
    we see that $d=6$ and $\lambda=1$. It is known that the only $A$-invariant curve of degree $6$ in $\PP^2$ is smooth, so $\widetilde{X}$ itself is smooth. However, \cite[Corollary 4.1.3]{RationalityII} implies that $\widetilde{X}$ is not rational, a contradiction.
\end{proof}

Thus, we are led to the case when $G$ acts faithfully on $\operatorname{Cl}(X)$, and by \cite{lmfdb} we have that $\mathrm{r}(X)\geqslant10$. In particular, the number of birational contractions $n$ performed in an MMP of $X$ satisfies $n\geqslant7$. Now we show that $X$ cannot contain surfaces of small degree.

\begin{lemma}
    \label{lemma:A6-ext-quadrics-g7}
    The $3$-fold $X$ does not contain quadric surfaces.
\end{lemma}

\begin{proof}
First, remember that the anticanonical embedding $X\hookrightarrow\PP^8=\PP(V)$ is given by a $9$-dimensional irreducible representation of $G$. Consider $S_1\subset X$ an irreducible quadric surface and write $S=\displaystyle\sum_{i=1}^mS_i$ for its $G$-orbit. Since $S$ is $G$-invariant we have that $S\sim a(-K_X)$ for some positive integer $a$ and we observe
$$2m=\deg(S)=(-K_X)^2\cdot S=a(-K_X)^3=12a.$$
Now, observe that all the components of $S$ are covered by lines, so if $D$ is the divisor spanned by all the lines on $X$ we have that $S\subset D$, and then Lemma \ref{lemma:divisor-of-lines-g7} implies that $a\leqslant7$. Considering this constraint, and taking into account the possible lengths for a $G$-orbit, we obtain that the possible cases are:
\begin{enumerate}
    \item $m=12$ and $a=2$,
    \item $m=30$ and $a=5$,
    \item $m=36$ and $a=6$.
\end{enumerate}
Before treating each case we state some general facts. Consider $H\in\vert-K_X\vert$ a general member, so $H$ is a smooth K3 surface and write 
$$S\vert_H=\displaystyle\sum_{i=1}^m S_i\vert_H=\sum_{i=1}^m C_i$$
where $C_i=S_i\vert_H$. Note that each $C_i$ is a smooth rational curve on $H$, so $C_i^2=-2$ for every $i\in\{1,\ldots,m\}$. Since $S\vert_H\sim aH\vert_H$ we compute that
$$C_1\cdot S\vert_H=C_1^2+\sum_{i=2}^mC_1\cdot C_i=S_1\cdot S\cdot H=a(-K_X)^2\cdot S_1=2a$$
and then 
\begin{equation}
    \label{equation:conics}
    \sum_{i=2}^mC_1\cdot C_i=2(a+1)
\end{equation}

We can also notice that $C_i\cdot C_j\leqslant2$ for every $1\leqslant i<j\leqslant m$. Indeed, assume that $C_i\cdot C_j\geqslant3$ and consider $\Pi_i=\langle C_i\rangle,\Pi_j=\langle C_j\rangle$ their linear spans, so $\Pi_i\simeq\Pi_j\simeq\PP^2$ and $C_i\cap C_j\subset\Pi_i\cap\Pi_j$. Since $C_i\cap C_j$ is a zero-dimensional subscheme of length $\geqslant3$ of each conic, we necessarily have that $\Pi_i=\Pi_j$. Now, notice that if $Q$ is a quadric that cuts out $X$, $Q\vert_{\Pi_i}$ is a conic in $\Pi_i$ that contains the degree $4$ subscheme $C_i\cup C_j$, so $Q\vert_{\Pi_i}=0$. Then we deduce that $\Pi_i\subset X$, but this contradicts the fact that $X$ does not contain planes.

Assume now that $m=12$ and $a=2$. In this case the stabilizer of $S_1$ is $G_1\simeq\mathfrak{A}_5$ and one can compute that the action of $G_1$ on the set of quadrics $S_2,\ldots,S_{12}$ is such that there is one orbit of length $5$ and one orbit of length $6$. Notice that in this case $S_1$ is smooth. Indeed, if $S_1$ were singular then it would be a cone over a plane conic curve with some vertex $P\in X$, and this vertex has to be fixed by $G_1$. However, one can compute that $H^0(X,\OOO_X(-K_X))$ splits as a sum of one $4$-dimensional irreducible and one $5$-dimensional irreducible representation of $G_1$. In particular, $G_1$ does not fix points on $\PP^8$ so every $S_i$ is smooth.

On the other hand, notice that $S_1$ is not Cartier. Indeed, in that case we would have $S_1\sim~\lambda(-K_X)$ for some $\lambda\in\mathbb{Z}_{\geqslant0}$ but then
$2=(-K_X)^2\cdot S_1=12\lambda$ so $\lambda=\frac{1}{6}$. This implies that $S_1$ contains at least one singular point of $X$, then $\operatorname{Sing}(X)\subset\operatorname{Supp}(S)$ and moreover, by symmetry each $S_i$ contains the same number of singular points. We know from Lemma \ref{lemma:A6-ext-gorenstein-g7-singular-points} that $\vert\operatorname{Sing}(X)\vert=20$. Denote by $p$ the number of singular points contained in $S_1$ and by $q$ the number of components of $S$ passing through a given singular point. We obtain the relation
$$12p=20q$$
which gives us that $p\in\{5,10,15,20\}$. Now, note that $S_1\simeq\PP^1\times\PP^1$ and its stabilizer $G_1\simeq\mathfrak{A}_5$ preserves both rulings and acts faithfully on at least one factor. Then, since $G_1$ does not have orbits of length $5,10$ or $15$ in $\PP^1$, we obtain that $p=20$. Then the quadric $S_1$ contains all the singular points of $X$ so its linear span $\langle \Sing(X)\rangle=\PP^3$ is $G$-invariant, which is a contradiction since $H^0(X,\OOO_X(-K_X))$ is $G$-irreducible.

Now we study the case $m=30$ and $a=5$. Here the stabilizer of each quadric is $G_i\overset{\textup{def}}{=}~\operatorname{Stab}_G(S_i)\simeq\mathfrak{S}_4$ and one can compute that the remaining quadrics in $\{S_1,\ldots,S_{30}\}$ split in orbits
$$\mathcal{O}_3(S_i),\;\;\mathcal{O}_6(S_i),\;\;\mathcal{O}_8(S_i),\;\;\mathcal{O}_{12}(S_i)$$
of length $3,6,8,12$ respectively under the action of $G_i$. Given that $G$ acts transitively on the set of quadrics this configuration is homogeneous, so most of the computations determined by these orbits will hold from the perspective of every $S_i$. In particular, the intersection numbers $C_i\cdot C_j$ will be constant in each orbit, and moreover these values are well-defined for pairs given that $S_j\in\mathcal{O}_k(S_i)$ if and only if $S_i\in\mathcal{O}_k(S_j)$. We denote this value by $x_k$ and notice that $0\leqslant x_k\leqslant2$ for every $k$.

On the other hand, one computes that the $G$-orbit $\{S_1,\ldots,S_{30}\}$ splits into two $A$-orbits $\mathcal{A}$ and $\mathcal{B}$ of length $15$ that are interchanged by $G\setminus A$, and for $S_i\in\mathcal{A}$ the $G_i$-orbits organize as
\begin{center}
    $\mathcal{O}_6(S_i),\mathcal{O}_8(S_i)\subset\mathcal{A}\;\;$ and $\mathcal{O}_3(S_i),\mathcal{O}_{12}(S_i)\subset\mathcal{B}$,
\end{center}
and symmetrically for $S_i\in\mathcal{B}$. Now, a crucial observation is that for $S_i,S_j$ different quadrics in the same $A$-orbit one can compute that
\begin{equation}
\label{equation:orbits-3}
    \vert\mathcal{O}_3(S_i)\cap\mathcal{O}_3(S_j)\vert=\begin{cases}
    1 \;\; \text{if} \;\; S_j\in\mathcal{O}_6(S_i)\\
    0 \;\; \text{if} \;\; S_j\in\mathcal{O}_8(S_i),
\end{cases} 
\end{equation}
Indeed, consider $S_i,S_j\in\mathcal{A}$ different quadrics. Denote by $m_k$ the number $\vert\mathcal{O}_3(S_i)\cap\mathcal{O}_3(S_j)\vert$ where $S_j\in\mathcal{O}_k(S_i)$ and consider the set
$$\mathcal{S}=\{(Q,R):Q\in\mathcal{O}_3(S_i),\; R\in\mathcal{O}_3(Q),\; R\neq S_i\}.$$
Counting the elements in $\mathcal{S}$ in two different ways gives $m_6=1$ and $m_8=0$. In a similar way one sees that
\begin{equation}
\label{equation:orbits-6}
    \vert\mathcal{O}_6(S_i)\cap\mathcal{O}_6(S_j)\vert=\begin{cases}
    1 \;\; \text{if} \;\; S_j\in\mathcal{O}_6(S_i)\\
    3 \;\; \text{if} \;\; S_j\in\mathcal{O}_8(S_i),
\end{cases} 
\end{equation}

By the equation \eqref{equation:conics} we see that
$$3x_3+6x_6+8x_8+12x_{12}=2(a+1)=12,$$
and it is clear that the only solutions for this equation are 
$$(x_3,x_6,x_8,x_{12})=(0,2,0,0),\;\; (2,1,0,0),\;\;(0,0,0,1).$$
Assume first that $x_6=2$ and set $\Delta=\displaystyle\sum_{S_i\in\mathcal{A}}C_i$. Given that each $C_i$ with $S_i\in\mathcal{A}$ meets $6$ members of $\mathcal{A}$ with multiplicity $2$ we have
$$\Delta\cdot H=15\cdot2=30, \;\; \Delta^2=15\cdot(-2)+15\cdot6\cdot2=150.$$
These computations contradict Hodge index theorem as $900=(\Delta\cdot H)^2\geqslant \Delta^2\cdot H^2=1800$.

Now assume that $x_3=2,x_6=1$ and $x_8=x_{12}=0$. If we write $\mathcal{O}_3(S_1)=\{S_\alpha,S_\beta,S_\gamma\}$, in this case we have that
$$C_1\cdot C_\alpha=C_1\cdot C_\beta=C_1\cdot C_\gamma=x_3=2,$$
and by \eqref{equation:orbits-3} it follows that
$$C_\alpha\cdot C_\beta=C_\alpha\cdot C_\gamma=C_\beta\cdot C_\gamma=x_6=1.$$
Hence the divisor $\Delta=C_1+C_\alpha+C_\beta+C_\gamma$ satisfies
\begin{center}
    $\Delta\cdot H=8\;\;$ and $\;\;\Delta^2=4\cdot(-2)+2(3\cdot x_3+3\cdot x_6)=10$,
\end{center}
but again $64=(\Delta\cdot H)^2<\Delta^2\cdot H^2=120$ contradicts Hodge index theorem.

Then we are led to assume $x_{12}=1$ and $x_3=x_6=x_8=0$. In this configuration we have that $C_i\cdot C_j=0$ except when $S_j\in\mathcal{O}_{12}(S_i)$. Suppose there is a linear relation
\begin{center}
    $\displaystyle\sum_{S_i\in\mathcal{A}}\alpha_iC_i+\displaystyle\sum_{S_j\in\mathcal{B}}\beta_jC_j=0\;\;$ in $\;\;\operatorname{Pic}(H)\otimes\QQ$.
\end{center}
Set $\alpha=\sum\alpha_i$ and $\beta=\sum\beta_j$. Taking intersection with $C_k$ for $S_k\in\mathcal{A}$ gives
$$-2\alpha_k+\beta-\sum_{S_j\in\mathcal{O}_3(S_k)}\beta_j=0,$$
and summing over $S_k\in\mathcal{A}$ we get
\begin{center}
    $-2\alpha+15\beta-3\beta=0\;\;$, i.e., $\;\;\alpha=6\beta$.
\end{center}
By symmetry we can also compute that $\beta=6\alpha$, so $\alpha=\beta=0$, and this means
\begin{center}
    $2\alpha_k=-\displaystyle\sum_{S_j\in\mathcal{O}_3(S_k)}\beta_j\;\;$ and $2\beta_l=-\displaystyle\sum_{S_i\in\mathcal{O}_3(S_l)}\alpha_i$. 
\end{center}
In consequence we get
$$4\alpha_k=\sum_{S_j\in\mathcal{O}_3(S_k)}\sum_{S_i\in\mathcal{O}_3(S_j)}\alpha_i=3\alpha_k+\sum_{S_i\in\mathcal{O}_6(S_k)}\alpha_i$$

Consider $A$ as the adjacency matrix of quadrics in $\mathcal{A}$ connected by an orbit of length $6$, i.e., 
\begin{center}
    $A_{ji}=1\;\;$ if $\;\; S_j\in\mathcal{O}_6(S_i)\;\;$ and $\;\;0\;\;$ otherwise for $\;\;1\leqslant i,j\leqslant15$,
\end{center}
and let $J$ be the matrix with all the entries equal to $1$. Then the vector $\alpha=(\alpha_k)$ satisfies $A\alpha=\alpha$. Now, noting that $(A^2)_{ji}=~\vert\mathcal{O}_6(S_i)\cap\mathcal{O}_6(S_j)\vert$ and that $J-I-A$ is the matrix marking the pairs of quadrics in $\mathcal{A}$ connected by an orbit of length $8$, using \eqref{equation:orbits-6} we obtain that
$$A^2=6I+A+3(J-I-A)=3I-2A+3J$$
Notice that if $\mathbf{1}$ is the vector of ones then $A\mathbf{1}=6\mathbf{1}$ and consider $W=\langle \mathfrak{1} \rangle^\perp\subset\QQ^{15}$. In this subspace we get
\begin{center}
    $A^2=3I-2A\;\;$ and then $\;\; (A-I)(A+3I)=0$,
\end{center}
so the eigenvalues of $A$ in $W$ are $\{1,-3\}$. Moreover, by definition the trace of $A$ is $0$, and this implies that the multiplicity of $1$ is $9$. Since the coefficients $\alpha_i$ in the linear relation determine $\beta_j$ and $A\alpha=\alpha$, we see that the subspace of linear relations between the $C_1,\ldots,C_{30}$ in $\QQ^{30}$ has dimension at most $9$. This means that $30-9=21>20\geqslant\rho(H)$, which is a contradiction.

Finally assume that $m=36$, $a=6$ and consider $G_1$ the stabilizer of $S_1$. In this case we obtain from \eqref{equation:conics} that
$$\sum_{i=2}^{36}C_1\cdot C_i=2(a+1)=14.$$
Then $G_1\simeq\mathrm{F}_5$ and $G_1$ decomposes $\{S_2,\ldots,S_{36}\}$ into orbits of length $5,10$ and $20$. These orbit lengths force the condition
$$\sum_{i=2}^{36}C_1\cdot C_i\equiv0\pmod{5},$$
which contradicts the equality above.
\end{proof}

\begin{lemma}
    \label{lemma:A6-ext-cubics-g7}
    The $3$-fold $X$ does not contain cubic surfaces, i.e., there is no irreducible surface $S\subset X$ with $(-K_X)^2\cdot S=3$.
\end{lemma}

\begin{proof}
Assume that $S_1$ is a cubic on $X$, i.e., $(-K_X)^2\cdot S_1=3$, and denote by $S=\displaystyle\sum_{i=1}^mS_i$ its $G$-orbit. Then
$$3m=\deg(S)=(-K_X)^2\cdot S=a(-K_X)^3=12a.$$
Now, consider $M=\langle S_1\rangle$ as the linear span of $S_1$, so $M$ can be either $\PP^3$ or $\PP^4$. In the former case, consider the restriction map 
    $$H^0(\PP^{8},\mathscr{I}_X(2))\longrightarrow H^0(M,\mathscr{I}_{S_1/M}(2))$$
Since $S_1$ is a cubic in $\PP^3$ we have that $H^0(M,\mathscr{I}_{S_1/M}(2))=0$, and this means that every quadric containing $X$ is zero along $M$. As we know from Lemma \ref{lemma:A6-ext-quadrics} that $X$ is cut out by quadrics this would imply $X=M$, a contradiction. Then $M=\PP^4$ and $S_1$ is a surface of minimal degree. In that case $S$ is covered by lines and by Lemma \ref{lemma:divisor-of-lines-g7} $a\leqslant7$. Hence $m=4a\leqslant28$ and the possible cases are
\begin{enumerate}
    \item $m=12$ and $a=3$,
    \item $m=20$ and $a=5$.
\end{enumerate}
Assume that $m=12$ and $a=3$. In this case we see that $S_1$ is smooth. Indeed, if $S_1$ is singular, it has to be the scroll $S(3,0)$ which is a cone over a twisted cubic curve, but this is not possible as the space $H^0(X,\OOO(-K_X))$ splits as a sum of a $4$-dimensional representation and a $5$-dimensional linear representations of $G_1$, so in particular $G_1$ does not fix points in $\PP^8$. Then $S_1$ is a Hirzebruch surface $S_1\simeq\FF_1$, so it contains a unique negative section $B\simeq\PP^1$ with $B^2=-1$ that has to be $G_1$-invariant. Now we prove that this section is a line in $\PP^8$. Recall that
\begin{center}
    $\Pic(S_1)=\mathbb{Z}[h]\oplus\mathbb{Z}[f]\;\;$ with $\;\; h^2=1,\; h\cdot f=1,\;\; f^2=0\;\;$ and $\;\;b\sim h-f$, 
\end{center}
where $b$ denotes the class of $B$ in $\Pic(S_1)$. The hyperplane section of $S_1$ will be given by $H=(-K_X)\vert_{S_1}$ and $H\sim\alpha h+\beta f$ for some $\alpha,\beta\in\mathbb{Z}_{>0}$. Then we derive that
$$3=\deg(S_1)=H^2=\alpha^2+2\alpha\beta,$$
and since $\alpha,\beta$ are positive we obtain $\alpha=\beta=1$, so $H\sim h+f\sim b+2f$. It follows that
$$H\cdot B=(b+2f)\cdot b=b^2+2f\cdot b=1,$$
in other words, $B$ is embedded in $\PP^8$ as a line. Since $H^0(X,\OOO(-K_X))$ contains no $2$-dimensional $G_1$-representations, this contradicts the fact that $B$ is $G_1$-invariant.

To finish the proof we have to analyze the case $m=20$ and $a=5$. In such a case the stabilizer of $S_1$ turns out to be $G_1\simeq\mumu_3^2\rtimes\mumu_4$. Let $H\in\vert-K_X\vert$ be a general member and consider $C_i=S_i\vert_H$, so $C_i^2=-2$ and $C_i\cdot H=3$ for every $i\in\{1,\ldots,20\}$. Similarly to the proof of Lemma \ref{lemma:A6-ext-quadrics-g7} we see that
$$\sum_{j=2}^{20}C_1\cdot C_j=3a+2.$$
Also, note that $C_i\cdot C_j\leqslant3$ for every $1\leqslant i,j\leqslant20$. Indeed, if $C_i\cdot C_j\geqslant4$, both cubics have the same linear span $\Lambda\simeq\PP^3$ and consider $U$ as the image of the space of defining quadrics of $X$ in $H^0(\Lambda,\OOO_{\Lambda}(2))$. Note that every non-zero quadric in $U$ is irreducible, as an irreducible non-planar cubic curve cannot lie in the union of two planes. If $\dim U\geqslant2$, we get two independent quadrics in $\Lambda$ containing the subscheme $C_i\cup C_j$ of degree $6$, impossible. If $\dim U=0$, we see that $\Lambda\subset X$, which is again not possible. Thus, $\dim U=1$ but this condition implies that the unique irreducible quadric in $\Lambda$ defined by $U$ is contained in $X$, contradicting Lemma~\ref{lemma:A6-ext-quadrics-g7}.

In addition, one sees that the set of cubics $\{S_1,\ldots,S_{20}\}$ splits into two orbits of length $1$ and two orbits of length $9$. Assume that $S_2$ is $G_1$-fixed and $\{S_3,\ldots,S_{11}\},\{S_{12},\ldots,S_{20}\}$ are the other $G_1$-orbits. Then
$$\sum_{j=2}^{20}C_1\cdot C_j=C_1\cdot C_2+9(C_1\cdot C_3+C_1\cdot C_{12})=17.$$
This implies that $C_1\cdot C_2\equiv8\pmod{9}$, but this is not possible as $0\leqslant C_1\cdot C_2\leqslant3$. We conclude the proof.
\end{proof}

By the previous Lemma we know that the minimal degree $d$ of a surface contracted by the morphisms $h_i$ in the diagram \eqref{equation:MMP} verifies $d\geqslant4$, so by \cite[Lemma 2.4]{Prokhorov-g-12} we obtain that
\begin{align*}
    (-K_{\widetilde{X}_n})^3\geqslant (-K_X)^3+(2d-2)n\geqslant 12+7\cdot 6=54
\end{align*}
Using the classification of smooth Fano $3$-folds we observe that $\rho(\widetilde{X}_n)=3$ is not possible as in such case $(-K_{\widetilde{X}_n})^3\leqslant52$. Now, if $\rho(\widetilde{X}_n)=2$, by the same inequality we obtain $(-K_{\widetilde{X}_n})^3\geqslant60$. In such case the only possibility would be $(-K_{\widetilde{X}_n})^3=62$ but then $\widetilde{X}_n$ would contain a plane, a contradiction. Then $\rho(\widetilde{X}_n)=1$, but again this inequality leads us to a contradiction. This contradiction shows that $g\neq7$.

\subsection{Fano $3$-folds of genus $g=9$}

In this section we assume that $g=9$. In particular, by Corollary \ref{corollary:A6-ext-g-7-9} we know that there exists a unique $G$-invariant divisor in $\vert-K_X\vert$.

\begin{lemma}
    \label{lemma:A6-ext-g9-Q-factorial}
    The $3$-fold $X$ is not $\QQ$-factorial. 
\end{lemma}

\begin{proof}
    Suppose first that $X$ is smooth. Then by \cite[Corollary 4.3.5]{KPS:Hilb} there is an embedding $\mathfrak{A}_6\hookrightarrow\Aut(\PP_C(\mathscr{E}))$ where $C$ is a smooth curve of genus $3$ and $\mathscr{E}$ is a simple vector bundle of rank $2$ over $C$. The natural projection $\PP_C(\mathscr{E})\to C$ gives an exact sequence
    $$1\longrightarrow\Aut(\PP_C(\mathscr{E})/C)\longrightarrow\Aut(\PP_C(\mathscr{E}))\longrightarrow\Aut(C),$$
    where $\Aut(\PP_C(\mathscr{E})/C)$ denotes the automorphisms of $\PP_C(\mathscr{E})$ acting fiberwise. Since $\mathfrak{A}_6$ is simple, it follows from the exact sequence that $\mathfrak{A}_6$ needs to act faithfully either on $C$ or on the general fiber of $\PP_C(\mathscr{E})\to C$. Both situations contradict Lemma \ref{lemma:A6-genus}.
    
    Now suppose $X$ is $\QQ$-factorial. Since $\rho(X)=1$ we can apply \cite[Theorem 1.3]{Prokhorov-G-Fanos} to obtain that $X$ has at most $3$ singular points. However, this means that every singular point of $X$ is $A$-fixed which is not possible due to Theorem \ref{theorem:A6-non-fixed}.
\end{proof}

\begin{lemma}
    \label{lemma:lines-g9}
    Let $V$ a smooth Fano threefold with $\Pic(V)=\mathbb{Z}[-K_V]$ and genus $g=9$. If $D$ denotes the divisor spanned by lines in $V$, then $D\sim-4K_V$.
\end{lemma}

\begin{proof}
    The proof is the same as that of Lemma~\ref{lemma:lines-g7}.
\end{proof}

\begin{lemma}
    \label{lemma:divisor-of-lines-g9}
    Let $D$ be the divisor spanned by all lines in $X$. Then $D\sim a(-K_X)$ for $a\in\{1,2,3,4\}$. 
\end{lemma}

\begin{proof}
    It is clear that $D\sim a(-K_X)$ for some positive integer $a$, as $D$ is $G$-invariant. The fact that $a\leqslant4$ comes from the same computation as in \cite[Lemma 2.6.9]{CheltsovPinardinProkhorov} using Lemma \ref{lemma:lines-g9}.
\end{proof}

 Thus it follows that $X$ is not $\QQ$-factorial. Let $f:\widetilde{X}\to X$ be a $\QQ$-factorialization. Then $\widetilde{X}$ is an almost Fano $3$-fold containing no planes, with $-K_{\widetilde{X}}\sim f^*(-K_X)$ and $\rho(\widetilde{X})=\mathrm{r}(X)>1$. Hence we are in the same situation as in diagram \eqref{equation:MMP}, whose notation we adopt from now on.

\begin{lemma}
    \label{lemma:720-764-g-9-faithful}
    The action of $G$ on $\operatorname{Cl}(X)$ is faithful.
\end{lemma}

\begin{proof}
    Assume that $A$ acts trivially on $\operatorname{Cl}(X)$. Moreover, notice that the unique $G$-invariant divisor $S\in\vert-K_X\vert$ is a smooth K3 surface because of Lemma \ref{lemma:A6-ext-invariant-surfaces}. Now, there is an injective homomorphism of $G$-modules $\operatorname{Cl}(X)\hookrightarrow\Pic(S)$ and since $A$ acts trivially on $\operatorname{Cl}(X)$, we should have that $\operatorname{Cl}(X)\hookrightarrow\Pic(S)^{\mathfrak{A}_6}$. Now, Lemma \ref{lemma:K3-a6} would imply that $\operatorname{Cl}(X)\simeq\mathbb{Z}$ which contradicts the fact that $X$ is not $\QQ$-factorial.
\end{proof}

As a consequence we get from the representation theory of $M_{10}$ (see \cite{lmfdb}) that $\mathrm{r}(X)\geqslant10$. In particular, running the MMP on $X$ one sees that the number of birational contractions will be $n\geqslant7$.

\begin{lemma}
    \label{lemma:720-764-quadrics-cubics-g9}
    The $3$-fold $X$ contains neither quadric nor cubic surfaces (i.e., there is no irreducible surface $S\subset X$ with $(-K_X)^2\cdot S=2$ or $(-K_X)^2\cdot S=3$).
\end{lemma}

\begin{proof}
    Assume that $S_1\subset X$ is an irreducible quadric and consider $S=\sum_{i=1}^mS_i$ its $G$-orbit. Then we have that $S\sim a(-K_X)$ for some positive integer $a$ and
    $$2m=\deg(S)=(-K_X)^2\cdot S=a(-K_X)^3=16a.$$
    Looking at the list of indices of subgroups of $G$, we see that $a\geqslant5$. On the other hand, since $S$ is covered by lines, if we consider $D$ as the divisor spanned by all lines in $X$ we should have that $S\subset D$ and Lemma \ref{lemma:divisor-of-lines-g9} implies that $a\leqslant4$. This contradiction shows that there are no quadric surfaces in $X$.

    Similarly, assume that $S_1$ is a cubic on $X$, i.e., $(-K_X)^2\cdot S_1=3$, and denote by $S=\sum_{i=1}^mS_i$ its $G$-orbit. In this case we obtain that
    $$3m=\deg(S)=(-K_X)^2\cdot S=a(-K_X)^3=16a,$$
    and in particular this implies that $a\geqslant9$.
    Now, consider $M=\langle S_1\rangle$ as the linear span of $S_1$, so $M$ can be either $\PP^3$ or $\PP^4$. In the former case, consider the restriction map 
    $$H^0(\PP^{10},\mathscr{I}_X(2))\longrightarrow H^0(M,\mathscr{I}_{S_1/M}(2))$$
    Since $S_1$ is a cubic in $\PP^3$ we have that $H^0(M,\mathscr{I}_{S_1/M}(2))=0$, and this means that every quadric containing $X$ is zero along $M$. As we know from Lemma \ref{lemma:A6-ext-quadrics} that $X$ is cut out by quadrics this would imply $X=M$, a contradiction. In the latter case, we see that $S_1$ would be a rational normal scroll, as it is a variety of minimal degree. In particular, $S$ is covered by lines and then Lemma \ref{lemma:divisor-of-lines-g9} gives $a\leqslant4$, a contradiction.
\end{proof}

By the previous Lemma we know that the minimal degree $d$ of a surface contracted by the morphisms $h_i$ in the diagram \eqref{equation:MMP} verifies $d\geqslant4$, so by \cite[Lemma 2.4]{Prokhorov-g-12} we obtain
\begin{align*}
    (-K_{\widetilde{X}_n})^3\geqslant (-K_X)^3+(2d-2)n\geqslant 16+7\cdot 6=58
\end{align*}
By the classification of smooth Fano $3$-folds we have that $\rho(\widetilde{X}_n)=3$ is not possible as in such case $(-K_{\widetilde{X}_n})^3\leqslant52$. Now, if $\rho(\widetilde{X}_n)=2$, by the same inequality we obtain $(-K_{\widetilde{X}_n})^3\geqslant64$, so $\rho(\widetilde{X}_n)=1$, but then this inequality leads us to a contradiction. Consequently, we obtain that $g\neq9$.

This argument concludes the proof of Theorem~\ref{theorem:A6-ext-Gorenstein}. 

\section{Proof of the Technical Theorem: non-Gorenstein case}
\label{section:A6-ext-non-Gorenstein}

In this section we prove the following theorem. 

\begin{theorem}
    \label{theorem:A6-ext-non-gorenstein}
    Let $X$ be a rational non-Gorenstein threefold with terminal singularities and let $G\subset\Aut(X)$ be a finite subgroup such that $X$ is a $G\QQ$-Fano threefold. Then $G\not\simeq M_{10}$.
\end{theorem}

Throughout this section, we assume that $X$ is a non-Gorenstein~rational $G\QQ$-Fano~\mbox{threefold}, where $G$ is a subgroup in $\mathrm{Aut}(X)$ such that $G\simeq M_{10}$, and we seek a contradiction.

\subsection{Non-Gorenstein points and the anticanonical linear system}
\label{subsection:non-gorenstein-anticanonical}

First, we start counting the singular points of $X$ and determining their singularity types.

\begin{proposition}
    \label{proposition:A6-ext-non-Gorenstein}
    Let $\Sigma$ be the set of all non-Gorenstein points of $X$. Then $\vert\Sigma\vert=12$, any point $P\in\Sigma$ is a cyclic quotient singularity of type $\frac{1}{2}(1,1,1)$ and $G$ acts transitively on $\Sigma$.
\end{proposition}

\begin{proof}
    By Lemma \ref{lemma:A6-ext-fixed}, we know that $\vert\Sigma\vert>10$. On the other hand, it follows from \cite[Corollary 2.4.6]{CheltsovPinardinProkhorov} that $\vert\Sigma\vert<16$. Thus, by Theorem \ref{theorem:A6-non-fixed} and looking at the indices of subgroups of $G$ we see that $\Sigma$ consists of a unique $G$-orbit of length $12$ and the stabilizer of a point $P\in\Sigma$ is $G_P\simeq\mathfrak{A}_5$. We also obtain from \cite{Reid:YPG} that the index of each $P\in\Sigma$ is $r=2$. Moreover, every point $P\in\Sigma$ is a cyclic quotient singularity, since the basket of singularities of $P$ contains just one ``virtual point''.
\end{proof}

\begin{corollary}
    \label{corollary:A6-ext-2K-cartier}
    The divisor $2K_X$ is Cartier.
\end{corollary}

\begin{proposition}
    \label{proposition:A6-ext-KX-dim}
    One has $\dim\vert-K_X\vert\neq0$.
\end{proposition}

\begin{proof}
    It follows from \cite[Corollary 2.4.7]{CheltsovPinardinProkhorov} and Proposition \ref{proposition:A6-ext-non-Gorenstein} that
    \begin{center}
        $\dim\vert-K_X\vert=\dfrac{1}{2}(-K_X)^3-1$, and in particular $\vert-K_X\vert$ is not empty.
    \end{center}
    Suppose $\dim\vert-K_X\vert=0$, so there exists a unique $S\in\vert-K_X\vert$ which is $G$-invariant and $(-K_X)^3=2$. Since $S$ is singular at every non-Gorenstein point, by Lemma \ref{lemma:A6-ext-invariant-surfaces} we have that $S$ is reducible and reduced, and it contains at least $10$ irreducible components. Consider $S^\prime$ an irreducible component of $S$ and let $m$ be the number of irreducible components of $S$. Since $2K_X$ is Cartier, we obtain that
    $$8=(-2K_X)^2\cdot S=m(-2K_X)^2\cdot S^\prime= m((-2K_X)^2\cdot S^\prime)\geqslant m\geqslant10,$$
    which is a contradiction.
 \end{proof}

 The following corollary follows directly from Corollary \ref{corollary:A6-ext-K3-surfaces}.

 \begin{corollary}
    \label{corollary:A6-ext-non-Gorenstein-k3}
     Let $S$ be a general surface in $\vert-K_X\vert$. Then $(X,S)$ has purely log terminal singularities and $S$ is a K3 surface with at most Du Val singularities.  
 \end{corollary}

\begin{lemma}
    \label{lemma:A6-ext-non-Gorenstein-A6-inv}
    The linear system $\vert-K_X\vert$ contains at most one $G$-invariant divisor.
\end{lemma}

\begin{proof}
    Suppose that $\vert-K_X\vert$ contains a pencil $\mathcal{M}$ of $G$-invariant surfaces, and let $S\in\mathcal{M}$ be a general member. By Corollary \ref{corollary:A6-ext-K3-surfaces}, $S$ is an $A$-invariant K3 surface with at most Du Val singularities, and it is singular since $X$ is non-Gorenstein. This contradicts Corollary \ref{corollary:K3-a6}.
\end{proof}

\begin{corollary}
    \label{corollary:A6-ext-KX-dim-2}
    One has $\dim\vert-K_X\vert\geqslant8$.
\end{corollary}

\begin{proof}
    The space $H^0(X,\OOO_X(-K_X))$ is a linear representation of the group $G$, and it contains at most one $1$-dimensional subrepresentation by Lemma \ref{lemma:A6-ext-non-Gorenstein-A6-inv}. On the other hand, Proposition \ref{proposition:A6-ext-KX-dim} implies that $h^0(X,\OOO_X(-K_X))\geqslant2$, and then from representation theory of $G$ (\cite{lmfdb}) we obtain that $h^0(X,\OOO_X(-K_X))\geqslant9$.
\end{proof}

Until now, we have observed that the linear system $\vert-K_X\vert$ is mobile, $X$ contains exactly $12$ non-Gorenstein points that form a $G$-orbit and all of them are cyclic quotient singularities of type $\frac{1}{2}(1,1,1)$. The linear system $\vert-K_X\vert$ gives a rational map $\psi:X\dashrightarrow\PP^n$, where
$$n=\dim\vert-K_X\vert=\frac{1}{2}(-K_X)^3-1\geqslant8.$$
Now consider the model $\overline{X}=\overline{\operatorname{im}(\psi)}\subset\PP^n$ and $\pi:Y\to X$ as the blow-up of all non-Gorenstein points of $X$, with $\pi$-exceptional surfaces $E_1,\ldots,E_{12}$. Then $\pi$ is $G$-equivariant with each $E_i\simeq\PP^2$,
$$(-K_Y)^3=(-K_X)^3-6=2n-4,$$
and in the following we study the properties of the anticanonical model $\overline{X}$.

\begin{lemma}
\label{lemma:720-764-base-points}
The base locus of the anticanonical linear system of $X$ is $\operatorname{Bs}\vert-K_X\vert=\Sigma$,
the divisor $-K_Y$ is big and nef, $\psi\circ\pi$ is the~morphism given by $\vert-K_Y\vert$ and the $G$-equivariant commutative diagram
$$
\xymatrix{
&Y\ar[ld]_{\pi}\ar[rd]^{\phi}&\\
X\ar@{-->}[rr]_{\psi}&&\overline{X},}
$$
is such that $\phi$ is crepant birational and $\overline{X}$ is a Fano threefold with canonical Gorenstein singularities such that $-K_{\overline{X}}\sim\OOO_{\mathbb{P}^n}(1)\vert_{\overline{X}}$.
\end{lemma}

\begin{proof}
    Consider $S\in\vert-K_X\vert$ a general member, so by Corollary \ref{corollary:A6-ext-non-Gorenstein-k3} $S$ is an irreducible K3 surface with at most Du Val singularities and let us denote by $\widetilde{S}$ its strict transform on $Y$. Thus 
    $$\widetilde{S}\sim_\QQ \pi^\ast(S)-\sum_{i=1}^{12}m_iE_i,$$
    where each $m_i\in\QQ_{>0}$ is such that $2m_i$ is an integer as the singularities of $X$ are of type $\frac{1}{2}(1,1,1)$. Then we get
    $$
    K_{Y}+\widetilde{S}\sim_{\mathbb{Q}} \pi^*(K_X+S)+\sum_{i=1}^{12}\Big(\frac{1}{2}-m_i\Big)E_i,
    $$
    and Corollary~\ref{corollary:A6-ext-positive-mobile} implies that $m_i=\frac{1}{2}$ for every $i$. In particular, we see from this relation that the linear system $\vert-K_Y\vert$ corresponds to the strict transform of $\vert-K_X\vert$, and moreover $\widetilde{S}$ is smooth along each $\pi$-exceptional surface as
    $$\widetilde{S}\vert_{E_i}\sim_\QQ (-K_Y)\vert_{E_i}\sim_\QQ -\frac{1}{2}E_i\vert_{E_i}\sim_\QQ \OOO_{E_i}(1).$$
    By construction the map $\phi=\psi\circ\pi:Y\dashrightarrow\PP^n$ is given by the linear system $\vert-K_Y\vert$. As the stabilizer $G_i$ of a plane $E_i$ is $G_i\simeq\mathfrak{A}_5$, $E_i$ contains neither $G_i$-invariant points nor lines. The previous computation also shows that $\vert-K_Y\vert$ does not contain any base point inside $E_1\cup\ldots\cup E_{12}$. Then the fact that $(-K_Y)^3=2n-4\geqslant12$ implies that $-K_Y$ is big and nef and the map $\phi$ is regular near each $E_i$.
    
    Thus, we see that $Y$ is a weak Fano threefold with at worst terminal Gorenstein singularities. By the basepoint-free theorem, for $m\gg 0$ the linear system $|-mK_Y|$ gives a morphism $\varphi\colon Y\to \overline{Y}$ such that $\overline{Y}$ is a Fano threefold with at worst canonical Gorenstein singularities and $-K_Y\sim \varphi^*(-K_{\overline{Y}})$. The morphism $\varphi$ is $G$-equivariant and the induced action on $\overline{Y}$ is faithful because $\varphi$ is birational. Now, by Lemma \ref{lemma:A6-bpf} we have that $\vert-K_{\overline{Y}}\vert$ is base point free, so $\vert-K_Y\vert$ also is. Moreover, by Lemma \ref{lemma:A6-ext-ample} $-K_{\overline{Y}}$ is very ample. Then projection formula gives
    $$
    \mathrm{dim}(|-K_Y|)=\mathrm{dim}(|-K_{\overline{Y}}|)=n
    $$
    and then we see that $\phi$ corresponds to the composition 
    $$Y\xrightarrow{\;\;\;\varphi\;\;\;}\overline{Y}\xrightarrow{\vert-K_{\overline{Y}}\vert}\PP^n$$
    and this implies that $\overline{X}\simeq\overline{Y}$. This argument concludes the proof.
\end{proof}

Set
$$g=\mathrm{g}(\overline{X})=\frac{(-K_{\overline{X}})^3+2}{2},$$
so we have that $n=g+1\geqslant8$. Furthermore, Lemma \ref{lemma:A6-ext-quadrics} implies that $\overline{X}$ is an intersection of quadrics in $\PP^n$.

\begin{lemma}
    \label{lemma:A6-ext-g-non-Gorenstein-10-12}
    One has $7\leqslant g\leqslant37$ and $g\notin\{8,10,12\}$.
\end{lemma}

\begin{proof}
    From Lemma \ref{lemma:A6-ext-non-Gorenstein-A6-inv} we know that $\vert-K_X\vert$ contains at most one $A$-invariant divisor. Hence the statement follows from Lemma~\ref{lemma:A6-ext-g-not-6-8} and \cite{Prokhorov-72}.
\end{proof}

\subsection{Terminal Gorenstein model}
\label{subsection:A6-ext-anticanonical-map-not-contracting-divisors}

We keep all the notation used in the previous part of this Section. The aim of this Subsection is to prove the following result about the singularities of the anticanonical model $\overline{X}$.

\begin{proposition}
    \label{proposition:terminal-model}
    The Fano threefold $\overline{X}$ has terminal Gorenstein singularities.
\end{proposition}

This Proposition will follow as a direct corollary from the next one, which we focus on proving during the rest of this Subsection. 

\begin{proposition}
    \label{proposition:A6-ext-small}
    The birational morphism $\phi$ is small.
\end{proposition}

Assume that $\phi$ is not small, denote by $F_1,\ldots,F_k$ the exceptional surfaces of $\phi$ and set $\overline{E}_i=\phi(E_i)$ for every $i\in\{1,\ldots,12\}$. We see that $\overline{E}_i$ are planes in $\PP^n$ and the restrictions $E_i\to\overline{E}_i$ are isomorphisms, so for every $i\neq j$ the intersection $\overline{E}_i\cap\overline{E}_j$ will be either empty, or a single point or a single line. First we study the images of $\phi$-exceptional divisors. 

\begin{lemma}
    For every $j\in\{1,\ldots,k\}$, the image $\phi(F_j)$ is a curve.
\end{lemma}

\begin{proof}
    Assume that $\phi(F_j)$ is a point, and notice that in such a case it is contained in some $\overline{E}_i$. Indeed, if that is not the case, $\pi(F_j)$ is a surface on $X$ that is not contracted by $\psi$, a contradiction. Thus $F_j$ must intersect one of the surfaces $E_1,\ldots,E_{12}$ on a curve. Since that curve is contracted by $\phi$, it intersects trivially with $-K_Y$, but it also needs to be contracted by $\pi$ so it has to intersect $-K_Y$ positively, a contradiction.
\end{proof}

In the following we set $C_i=\phi(F_i)$ and $d_i=\operatorname{deg}(C_i)$ for every $i\in\{1,\ldots,k\}$. Note that it could be possible to have $\phi(F_i)=\phi(F_j)$ for some $i\neq j$.

\begin{lemma}
    \label{lemma:curves-19}
    One has 
    \begin{align*}
    \sum_{i=1}^{k}d_i\leqslant19.
\end{align*}
\end{lemma}

\begin{proof}
    Let $\overline{H}$ be a general hyperplane section of $\overline{X}$ and let $H$ be its strict transform on $Y$. In this setting $\overline{H}$ has Du Val singularities, $H$ is a smooth K3 surface and moreover, each $F_i\vert_H$ is a disjoint union of $d_i$ rational curves on $H$ that are contracted by the morphism $\phi\vert_H:H\to\overline{H}$. As $\operatorname{rk}\Pic(H)\leqslant20$, we get the stated bound.
\end{proof}

Since $X$ is a $G\QQ$-Fano threefold, it follows that $\mathrm{r}(\overline{X})^G=1$. Using that the divisor
$$\overline{D}=\sum_{i=1}^{12}\overline{E}_i$$
is $G$-invariant, we can compute that $\overline{D}\sim_{\QQ}\frac{6}{g-1}(-K_{\overline{X}})$, and the ampleness of $-K_{\overline{X}}$ implies that for some $i,j\in\{1,\ldots,12\}$ with $i\neq j$ the intersection $\overline{E}_i\cap\overline{E}_j$ will be a line. Subsequently, we obtain

\begin{lemma}
    If $E_i\cap F_j\neq\varnothing$, then $E_i\cap F_j$ is a line in $E_i\simeq\PP^2$. 
\end{lemma}

\begin{proof}
    Suppose that $Z=E_1\cap F_1\neq\varnothing$. Then $Z$ is a curve of degree $d$ in $E_1\simeq\PP^2$, and then $C_1=\phi(Z)$ is a curve of the same degree in $\overline{E}_1\simeq\PP^2$. Note that if $d\neq1$ then $C_1$ is not contained in any other plane among $\overline{E}_2,\ldots,\overline{E}_{12}$, so the $G$-orbit of $C_1$ contains at least $12$ curves of degree $d\geqslant2$, contradicting \ref{lemma:curves-19}. 
\end{proof}

\begin{lemma}
    \label{lemma:A6-ext-non-gor-lc}
    Every line among $C_1,\ldots,C_k$ is contained in at least $3$ planes among $\overline{E}_1\ldots,\overline{E}_{12}$. In particular, the log pair $(\overline{X},\overline{D})$ is not log canonical at the general point of every line $C_1,\ldots,C_k$.
\end{lemma}

\begin{proof}
    Assume that $C_1$ is contained in $\overline{E}_1$. The stabilizer $G_1$ of the plane $\overline{E}_1$ is $G_1\simeq\mathfrak{A}_5$. Since $\overline{E_1}$ does not have $G_1$-invariant lines or points, the $G_1$-orbit of $C_1$ consists of at least $5$ lines in $\overline{E}_1$. If $C_1$ is contained in at most $2$ planes among $\overline{E_1},\ldots,\overline{E}_{12}$, the $G$-orbit of $C_1$ would have at least $\frac{12\cdot 5}{2}=30$ lines, but this would contradict \ref{lemma:curves-19}. The last assertion is a consequence of \cite[Proposition 16.6]{Utah}.
\end{proof}

Set $\lambda=\operatorname{lct}(\overline{X},\overline{D})$. Since $\overline{X}$ has canonical singularities, it is klt. By Lemma~\ref{lemma:A6-ext-non-gor-lc} we see that $0<\lambda<1$ and $(\overline{X},\lambda\overline{D})$ is log canonical. Let $C$ be a minimal log canonical center of $(\overline{X},\lambda\overline{D})$. Since $\overline{D}$ is reduced and $\lambda<1$ we have that $\dim C\leqslant1$. 

Pick a rational number $\mu>1$ such that $\lambda\mu<1$. Since $\overline{D}$ is ample, \cite[Lemma 2.8]{cheltsov2011exceptional} gives the existence of a $G$-invariant effective $\QQ$-divisor $B$ such that $B\sim_\QQ\lambda\mu\overline{D}$ such that $(\overline{X},B)$ is log canonical and
$$\operatorname{Nklt}(\overline{X},B)=\bigcup_{g\in G}g(C).$$
Furthermore,
$$-(K_{\overline{X}}+B)\sim_\QQ\left(1-\frac{6\lambda\mu}{g-1}\right)(-K_{\overline{X}})$$
is ample, since $g\geqslant7$ and $\lambda\mu<1$. By minimality of $C$ all the translates $g(C)$ for $g\in G$ are disjoint, but then Koll\'ar-Shokurov connectedness implies that $C$ is $G$-invariant. Choose a plane $\overline{E}_i$ containing $C$. Since $G$ acts transitively on the planes $\overline{E}_1,\ldots,\overline{E}_{12}$, $C$ is contained in all such planes, which forces $C$ to be either a line or a point. In any case the stabilizer $G_i\simeq\mathfrak{A}_5$ leaves $C$ invariant inside $\overline{E}_i$, contradicting the fact that $\overline{E}_i$ contains neither $G_i$-fixed points nor $G_i$-invariant lines. This argument finishes the proof of Proposition~\ref{proposition:A6-ext-small}.

\subsection{End of the proof of Technical Theorem}
\label{subsection:non-gorenstein-reduction-A6-ext}

We continue using the notations and assumptions from Subsection~\ref{subsection:A6-ext-anticanonical-map-not-contracting-divisors}. From the previous section we know that $\phi:Y\to\overline{X}$ is a small birational morphism, and then $\overline{X}$ is a terminal Gorenstein Fano threefold. Then, we can consider its $\QQ$-Gorenstein smoothing $V$, i.e., $V$ is a smooth Fano threefold such that
\begin{center}
    $(-K_V)^3=(-K_{\overline{X}})^3=2g-2\geqslant12,\quad \rho(V)=\rho(\overline{X}),\quad$ and $\quad\iota(V)=\iota(\overline{X})=1$.
\end{center}
In particular, from \cite{Namikawa} we obtain the following bound on the number of singular points of $\overline{X}$
\begin{equation}
    \label{equation:28}
    \vert\operatorname{Sing}(\overline{X})\vert\leqslant20+h^{1,2}(V)-\rho(V)\leqslant19+h^{1,2}(V)\leqslant28.
\end{equation}

We now make the following maximality assumption on $X$. Note that, if it were necessary, replacing $X$ by another non-Gorenstein $G\QQ$-Fano threefold $X^\prime$ that is $G$-birational to $X$, we assume that
\begin{equation}
\label{equation:maximal}
\dim(\vert-K_X\vert)\geqslant\dim(\vert-K_{X^\prime}\vert)
\end{equation}
for any non-Gorenstein $G\QQ$-Fano threefold $X^\prime$ that is $G$-birational to $X$. 

\begin{lemma}
    \label{lemma:not-iso}
    The birational morphism $\phi$ is not an isomorphism.
\end{lemma}

\begin{proof}
    Suppose $\phi$ is an isomorphism, so $Y$ is a terminal Gorenstein Fano threefold. Set $E~=~E_1+~\ldots~+~E_{12}$. Since all the non-Gorenstein points form one $G$-orbit, $E$ is $G$-irreducible and $\rho(Y)^G=2$. One extremal ray of the $G$-invariant Mori cone corresponds to $\pi$, so consider $h:Y\to X^\prime$ as the extremal contraction given by the other $K_Y$-negative extremal ray. By \hyperref[theorem:fibration]{Fibration Theorem} we have that $h$ is birational, and it cannot be small since from \cite[Theorem 4.7]{KM:92} we know that a threefold with terminal Gorenstein singularities does not admit such a contraction.
    In consequence $h$ has to contract a $G$-invariant divisor $F$, and $X^\prime$ is a $G\QQ$-Fano threefold which is non-Gorenstein due to Section~\ref{section:A6-ext-gorenstein}. If $h(F)$ is a curve, then it is contained in the base locus of the strict transform of $\vert-K_Y\vert$ on $X^\prime$, which is contained in $\vert-K_{X^\prime}\vert$, but since $X^\prime$ has only finitely many non-Gorenstein points, we get that 
    $$\dim(\vert-K_X\vert)=\dim(\vert-K_Y\vert)<\dim(\vert-K_{X^\prime}\vert),$$
    which contradicts our maximality assumption~\eqref{equation:maximal}. Then $h(F)$ has to be the $G$-orbit of a point. We claim that $F$ meets some plane $E_i$. Assume that this is not the case and write
    $$F\equiv a\pi^\ast(-K_X)-bE$$
    for some $a,b\in\mathbb{Q}$ with $a>0$. If $\ell_i\subset E_i$ is a line then $F\cdot\ell_i=2b$, and by disjointness $b=0$. Then we obtain that $F\sim a\pi^\ast(-K_X)$, which is impossible as $F$ cannot be nef. Without loss of generality assume that $F\cap E_1\neq\varnothing$. Moreover, this intersection has to be a curve since $Y$ is smooth along $E_1$ and $E_1$ and $F$ do not share any irreducible component. However, $h\vert_{E_1}$ cannot contract a curve as in that case we are contracting the whole plane $E_1$. This contradiction proves the statement.
\end{proof}

\begin{lemma}
\label{lemma:2-planes-through-point-A6-ext}
Let $P$ be a~point in $\overline{X}$. Then at most two planes among $\overline{E}_1,\ldots,\overline{E}_{12}$ contain $P$.
\end{lemma}

\begin{proof}
    Assume that three planes among $\overline{E}_1,\ldots,\overline{E}_{12}$ contains $P$ and let us seek a~contradiction. Without loss of generality, we may assume that $P=\overline{E}_1\cap\overline{E}_2\cap \overline{E}_3$. Let $T_P\subset\mathbb{P}^n$ be the~embedded tangent space to $\overline{X}$ at  $P$. Then $T_P\simeq\mathbb{P}^4$, and
    $$
    \overline{E}_1\cup\overline{E}_2\cup\overline{E}_3\subset T_P\cap\overline{X}.
    $$
    Set $Z=T_P\cap\overline{X}$. Since $\overline{X}$ is an intersection of quadrics, $Z$ is an~intersection of quadrics, and dimensions of its irreducible components are at most two. Now, intersecting $Z$ with a~general hyperplane $H\subset T_P$, we obtain three skew lines $H\cap \overline{E}_1$, $H\cap \overline{E}_2$, $H\cap \overline{E}_3$ in $H\simeq\mathbb{P}^3$ (note that $\overline{E}_1\cap\overline{E}_2=\{P\}$ since $\phi$ is small) such that they are contained in the~intersection of quadrics $H\cap Z$, whose irreducible components are at most one-dimensional. But this is impossible, since three skew lines in $\mathbb{P}^3$ are contained in a~unique quadric.
\end{proof}

Now, consider $k=\vert\operatorname{Sing}(\overline{X})\cap\overline{E}_{1}\vert$ and let $G_1\simeq\mathfrak{A}_5$ be the stabilizer of $\overline{E}_1$ in $G$. Note that $k\neq0$ because of Lemma \ref{lemma:not-iso}. Since $G_1$-orbits on $\overline{E}_i$ have a minimum length of $6$, we obtain $k\geqslant6$. By Lemma \ref{lemma:2-planes-through-point-A6-ext} any singular point of $\overline{X}$ is contained in at most two planes among $\overline{E}_1,\ldots,\overline{E}_{12}$, and then the threefold $\overline{X}$ contains at least $\frac{12\cdot k}{2}\geqslant36$ singular points. This contradicts \eqref{equation:28} and concludes the proof of Theorem \ref{theorem:A6-ext-non-gorenstein}. Hence, we conclude the proof of \hyperref[techtheorem]{Technical Theorem}.

\section{Proof of the Characterization Theorem}
\label{section:chartheorem}

In this section we let $G$ be isomorphic to either $\PGL_2(\FF_7)$ or $\PGL_2(\FF_9)$. By \hyperref[theorem:fibration]{Fibration Theorem}, it is enough to consider their actions on $G\QQ$-Fano threefolds. We will prove that every such Fano threefold is $G$-birational to the threefold \eqref{equation:W}. Then Proposition~\ref{proposition:W-actions} and Remark~\ref{remark:PGL27-flag} will give the uniqueness of the corresponding conjugacy classes in $\Cr_3(\CC)$.

The argument for $\PGL_2(\FF_9)$ is compared with Sections~\ref{section:A6-ext-gorenstein} and \ref{section:A6-ext-non-Gorenstein}. For $\PGL_2(\FF_7)$ we use the arguments of \cite[Part 2]{CheltsovPinardinProkhorov}, recording the changes required by the action of the outer involution and by working over $\CC$. The additional anticanonical models of genus $4$ that appear in this case are treated separately below.

\subsection{The group $\PGL_2(\FF_9)$}

Let $G\simeq\PGL_2(\FF_9)$ and $A\simeq\mathfrak{A}_6$ be its socle. Suppose that $X$ is a rational $G\QQ$-Fano threefold that is not $G$-birational to \eqref{equation:W}.

All the arguments concerning actions of $A$ presented in Subsection \ref{subsection:A6} remain available. For the arguments involving the whole group, one needs to distinguish representations on spaces of sections from representations on divisor classes. The former are complex representations, whereas
$$\Pic(X)\otimes\QQ,\qquad \operatorname{Cl}(X)\otimes\QQ$$
are rational representations. Although $G$ has four non-isomorphic complex representations of dimension $8$, none of them is defined over $\QQ$. The minimum dimension of a faithful rational representation of $G$ is $9$. Consequently, if the action on the space of divisor classes is faithful, its rank is at least $10$, since the canonical class spans an invariant line. Thus the rank estimates used in Sections~\ref{section:A6-ext-gorenstein} and \ref{section:A6-ext-non-Gorenstein} are unchanged.

The additional anticanonical representation of dimension $8$ is excluded as follows.

\begin{lemma}
    \label{lemma:PGL29-g-6}
    Let $X$ be a Fano threefold with canonical Gorenstein singularities such that it admits a faithful action of $G$. Then $g\neq6$.
\end{lemma}

\begin{proof}
    By Lemmas \ref{lemma:A6-bpf} and \ref{lemma:A6-hyperelliptic}, $-K_X$ is very ample. Suppose that $g=6$. We identify $X$ with its anticanonical image in $\PP^7$. Then $V=H^0(X,\OOO_X(-K_X))$ is an irreducible $8$-dimensional $G$-representation and $V$ is $A$-irreducible. By \cite[Corollary 5.5]{P:JAG:simple} $\dim H^0(\PP^7,\mathscr{I}_X(2))=6$, so $A$ acts trivially on this space. This gives a contradiction, as by Schur's Lemma one has $\dim(\Sym^2V)^A \leqslant1$.
\end{proof}

Therefore, the Gorenstein argument given in Section \ref{section:A6-ext-gorenstein} runs and gives $W$ as the only possible rational model. If $X$ is non-Gorenstein, the same approach and proofs presented in Section~\ref{section:A6-ext-non-Gorenstein} lead us to a numerical contradiction. Thus, we obtain the conclusion of the \hyperref[auxtheorem]{Characterization Theorem} for $\PGL_2(\FF_9)$.

\subsection{Complete intersections of a quadric and a cubic}
\label{subsection:irrationality}

In this section we consider $G\simeq\PGL_2(\FF_7)$ with socle $A\simeq\PSL_2(\FF_7)$. As noted in the Introduction, there exist exactly two such $G$-invariant complete intersections of a quadric and a cubic in $\PP^5$ and it can be verified, e.g. using MAGMA, that they are smooth. Here we prove, following the method of \cite{BeauvilleA7}, that such complete intersections are irrational.

\begin{theorem}
\label{theorem:X-2-3-irrational}
    Let $X\subset\PP^5$ be a smooth $\PGL_2(\FF_7)$-invariant complete intersection of a cubic hypersurface and a quadric hypersurface. Then $X$ is irrational.
\end{theorem}

\begin{proof}
The group $A$ has unique irreducible faithful representations $V_d$ of dimension $d\in\{6,8\}$, and two irreducible $3$-dimensional ones denoted by $V_3$ and $V_3^\vee$. The space of sections $V=H^0(X,\OOO_X(-K_X))$ is $G$-irreducible of dimension $6$, and its restriction to $A$ is either $V_6$ or $V_3\oplus V_3^\vee$. The latter possibility is excluded as
$$(\Sym^3(V_3\oplus V_3^\vee))^A=0,$$
so $V\vert_A\simeq V_6$. Let $JX$ be the intermediate Jacobian of $X$. By definition of $X$ and the Euler exact sequence we get that the sequence
$$
0\to V\to H^0(X,\mathscr{O}_X(1))\oplus H^0(X,\mathscr{O}_X(2))\to T_0(JX)\to0
$$
is exact, and then 
$$T_0(JX)\simeq\Sym^2 V_6/\CC\simeq V_6^{\oplus2}\oplus V_8$$
as $A$-modules. In particular $\dim JX=20$.

Assume that $X$ is rational. By the Clemens-Griffiths criterion \cite[Corollary 3.26]{CG72}, $JX$ is a product of Jacobians as a principally polarized abelian variety, i.e., there is an isomorphism 
$$JX\simeq \operatorname{Jac}(C_1)^{a_1}\times\cdots\operatorname{Jac}(C_m)^{a_m}$$
that induces an isomorphism of $A$-modules
$$
T_0(JX) \simeq T_0(\operatorname{Jac}(C_1)^{a_1})\oplus\cdots\oplus T_0(\operatorname{Jac}(C_m)^{a_m}),
$$
and the group $A$ permutes its indecomposable factors. For each orbit of factors, their product determines an $A$-submodule of dimension $d\in\{6,8,12,14,20\}$, and since $20$ is not a sum of integers in $\{8,14\}$, at least one orbit has $d\in\{6,12,20\}$. If it contains $r$ factors, then $r\mid d$. As the only proper subgroup indices of $A$ at most $20$ are $7,8,14$, we get $r=1$. Then an individual Jacobian of genus $6,12$ or $20$ is $A$-stable so it has a faithful $A$-action. By the Torelli theorem we get that $A$ acts faithfully on the corresponding curve, but this contradicts \cite[Lemma 2.5.1]{CheltsovPinardinProkhorov}.
\end{proof}

\subsection{Gorenstein models for $\PGL_2(\FF_7)$}
\label{subsection:PGL27-gorenstein}

Throughout this subsection, we set $G=\PGL_2(\FF_7),A=\PSL_2(\FF_7)$ and choose an involution $\tau\in G\setminus A$. We follow the Gorenstein argument of Section~\ref{section:A6-ext-gorenstein}, referring to its proofs or to \cite{CheltsovPinardinProkhorov} whenever they apply with the same hypotheses. We point out the changes required to adapt the argument to this group.

\begin{proposition}
    \label{proposition:PGL27-gorenstein}
    Let $X$ be a rational $G\QQ$-Fano threefold with terminal Gorenstein singularities. Then $X$ is $G$-birational to the threefold \eqref{equation:W}.
\end{proposition}

We first establish a useful lemma and then a main reduction step needed to prove this statement.

\begin{lemma}
\label{lemma:42-fixed-point}
Let $X$ be a threefold with terminal singularities that admits a faithful action of the group $\Gamma=\mathrm{F}_7$. Then $\Gamma$ does not fix points in $X$.
\end{lemma}

\begin{proof}
Suppose that $\Gamma$ fixes a point $P\in X$. Observe that $\Gamma\cap A$ is a subgroup $\mumu_7\rtimes \mumu_3$. Hence, it follows from \cite[Lemma 2.5.11]{CheltsovPinardinProkhorov} that either $X$ is smooth at $P$, or $X$ has a~cyclic quotient singularity of type $\frac{1}{2}(1,1,1)$ at $P$. If $X$ is smooth at $P$, then $\Gamma$ acts faithfully on the Zariski tangent space of $X$ at $P$, which is impossible, since $\Gamma$ does not have faithful $3$-dimensional representations \cite{GroupNames}. Thus, $X$ has a~cyclic quotient singularity of type $\frac{1}{2}(1,1,1)$ at $P$. 

We can regard $X$ as a small neighborhood of $P$. Let $\pi\colon \widetilde{X}\to X$ be a~canonical index-$1$ cover, let $\widetilde{P}\in\widetilde{X}$ be the preimage of the point $P$, and let $\widetilde{\Gamma}$ be the~natural lift of the group $\Gamma$ in $\mathrm{Aut}(\widetilde{X})$. Then $\widetilde{X}$ is smooth at $\widetilde{P}$, and $X=\widetilde{X}/\mumu_2$, where  the~action of $\mumu_2$ on $\widetilde{X}$ is free outside of $\widetilde{P}$. Thus, the group $\widetilde{\Gamma}$ is a central extension of $\Gamma$ by $\mumu_2$, so it follows from \cite{GroupNames} that either $\widetilde{\Gamma}\simeq\Gamma\times\mumu_2$ or $\widetilde{\Gamma}\simeq\mumu_7\rtimes\mumu_{12}$. As before, we see that the group $\widetilde{\Gamma}$ acts faithfully on the Zariski tangent space of $\widetilde{X}$ at $\widetilde{P}$, which is impossible, since $\widetilde{\Gamma}$ does not have faithful $3$-dimensional representations \cite{GroupNames}.
\end{proof}

\begin{lemma}
\label{lemma:PGL27-gorenstein}
    Suppose that $X$ satisfies the assumptions of Proposition \ref{proposition:PGL27-gorenstein} and $X$ is not $G$-birational to $W$. Then 
    \begin{center}
        $\Pic(X)=\mathbb{Z}[-K_X]\quad$ and $\quad g(X)\in\{6,10\}$.
    \end{center}
    Moreover, $-K_X$ is very ample, $X$ is an intersection of quadrics, $\vert-K_X\vert$ contains no $A$-invariant divisor and $X$ is not $\QQ$-factorial.
\end{lemma}

\begin{proof}
    Neither $G$ nor its subgroup $\mathrm{F}_7$ has a faithful $4$-dimensional representation. The tangent space arguments used in Section~\ref{section:A6-ext-gorenstein} therefore exclude the existence of $G$-fixed point and $G$-orbits of length $8$ on $X$. It also shows that $G$ does not embed into $\PGL_4(\CC)$ (see \cite{lmfdb}).

    The index reduction of Section~\ref{section:A6-ext-gorenstein} now applies using \cite[Lemma 2.5.19, Remark 2.5.17]{CheltsovPinardinProkhorov}: apart from $W$, the index-two cases have either an intrinsic fixed point or an equivariant double cover or birational contraction to $\PP^3$.

    To prove $\rho(X)=1$, we can follow the argument in the proof of Lemma \ref{lemma:A6-ext-pic}. Most of the proof remains the same, but cases with $\rho(X)=2$ are slightly different. The quadric case is impossible by \cite[Remark 2.5.17]{CheltsovPinardinProkhorov}. In case (1.2.1) of \cite{P:GFano2}, the argument given in the proof of \cite[Corollary 2.6.6]{CheltsovPinardinProkhorov} leaves only the case of a smooth $A$-invariant (2,2)-divisor in $\PP^2\times\PP^2$, which is irrational \cite{AB92}. For the intersection of three (1,1)-divisors in $\PP^3\times\PP^3$, the projection argument produces  a unique $A$-orbit of length $8$ on $X$, again impossible. The remaining case is $X\simeq W$, excluded by assumption. Thus $\rho(X)=1$. 

    Put $H=-K_X$ and $V=H^0(X,\OOO_X(H))$. By \cite[Theorem 2.2]{Prokhorov-G-Fanos}, $\vert H\vert$ is base point free, and the anticanonical morphism $\phi:X\to Y$ is finite of degree at most two. By the classification in \cite[Lemma 2.5.24]{CheltsovPinardinProkhorov}, the induced $G$-action on $Y$ is faithful. Consequently, $V$ is faithful and $g+2=\dim V\geqslant6$. Again by \cite[Theorem 2.2]{Prokhorov-G-Fanos} we get that $\phi$ is an embedding, and its image is an intersection of quadrics in case that $g\geqslant5$. Theorem \ref{theorem:X-2-3-irrational} excludes the case $g=4$.

    If $g=5$, on the one hand we would have that $\dim(\Sym^2V)^A \leqslant2$, but on the other hand by \cite[Corollary 5.5]{P:JAG:simple} there should be a $3$-dimensional $A$-invariant subspace of quadrics. Thus $g\neq5$. We next exclude $g=12$. Indeed, if $g=12$ by \cite[Theorem 1.3]{Prokhorov-g-12} either $X$ is smooth or $\Sing(X)$ consists of a unique ordinary double point. The latter case is not possible as we cannot have a $G$-fixed point, so $X$ has to be smooth. Then by \cite[Lemma 4.3.4]{KPS:Hilb} $G$ acts faithfully on the Hilbert scheme of conics of $X$, but this scheme is isomorphic to $\PP^2$, a contradiction. Thus $g\neq12$.

    Hence $6\leqslant g\leqslant10$. We next apply the invariant-divisor argument used in Section~\ref{section:A6-ext-gorenstein}. Combining the arguments used in the proof of Lemma \ref{lemma:A6-ext-Gorenstein-G'-inv} together with the adjunction argument used in the proof of Lemma \ref{lemma:A6-ext-invariant-surfaces} or \cite[Lemma 2.5.37]{CheltsovPinardinProkhorov}, one proves that every $G$-invariant member of $\vert H\vert$ is a smooth K3 surface with a faithful action of $G$. Subsequently, one can prove that the linear system $\vert H\vert$ can contain at most one $G$-invariant member, and then by representation theory of $G$ one has that $g\notin\{7,8,9\}$. We get that the only possibilities are $g\in\{6,10\}$.

    To finish the proof, assume that $X$ is $\QQ$-factorial. If $g=10$, it follows from \cite[Theorem 1.3]{Prokhorov-G-Fanos} that $X$ has at most two singular points, so each singular point is $A$-fixed. As this contradicts Lemma~\ref{lemma:42-fixed-point}, we obtain that $g=6$. However, in this case the same proof of \cite[Lemma 2.6.16]{CheltsovPinardinProkhorov} applies to see that this case is also not possible. Thus, $X$ is not $\QQ$-factorial.
\end{proof}

In the following we consider $f:\widetilde{X}\to X$ a $\QQ$-factorialization of $X$ and we run the MMP, obtaining the same diagram \eqref{equation:MMP}. We keep the same notation.

Assume first that $g=10$. The orbit-sum argument given in Lemma \ref{lemma:A6-ext-g-7-div} implies that the degree of every irreducible surface in $\widetilde{X}$ is divisible by $3$. The proofs of \cite[Lemmas 2.6.9 and 2.6.11-2.6.13]{CheltsovPinardinProkhorov} exclude the possibilities of divisorial extremal contractions and del Pezzo fibrations, leaving an $A$-equivariant conic bundle $h:\widetilde{X}\to\PP^2$. Then, we can construct the same diagram \eqref{equation:conic-bundles}, and repeating the same computation as there one sees that the discriminant curve of $h$ has degree $d=3$. This result contradicts \cite[Remark 2.4.20]{CheltsovPinardinProkhorov}, then one concludes that the case $g=10$ is impossible.

Let $g=6$. Then the same argument as in the proof of \cite[Lemma 2.6.14]{CheltsovPinardinProkhorov} shows that $\vert H\vert$ contains no $A$-invariant member. On the other hand, one can note that a similar argument to Lemma \ref{lemma:A6-ext-g-7-div} shows that the degree of every irreducible surface in $\widetilde{X}$ is divisible by $5$; moreover, an $A$-invariant surface has degree at least $10$. If $A$ acted trivially on $\operatorname{Cl}(X)$, the diagram \eqref{equation:MMP} would be $A$-equivariant, and then the proof of \cite[Lemma 2.6.22]{CheltsovPinardinProkhorov} or Lemma \ref{lemma:A6-ext-g-7-faithful} using these degree bounds would give a contradiction. Thus $G$ acts faithfully on $\operatorname{Cl}(X)$ and $\rho(\widetilde{X})\geqslant7$. Now the proof of \cite[Lemmas 2.6.23-2.6.24]{CheltsovPinardinProkhorov} shows that $Z$ is a point and that we have
 $$64\geqslant(-K_{\widetilde{X}_n})^{3}\geqslant10+\sum_{i=1}^n(2(-K_{\widetilde{X}})^2\cdot\widetilde{E}_i-2)\geqslant10+8n.$$
Therefore $n=6,\rho(\widetilde{X})=7$ and $(-K_{\widetilde{X}})^2\cdot\widetilde{E}_i=5$ for every $i=1,\ldots,6$. Now if we consider the subgroup $\Gamma\simeq\mumu_7\rtimes\mumu_3\subset A$, the character restriction in \cite[Lemma 2.6.25]{CheltsovPinardinProkhorov} gives $\operatorname{rk}\operatorname{Cl}(X)^\Gamma=1$. Then the reduced $\Gamma$-orbit sum of $f(\widetilde{E}_1)$ has $N$ irreducible components and class $aH$ for some $a\in\mathbb{Z}_{>0}$. Taking degrees one gets
\begin{center}
    $5N=10a,\quad$ but $\quad N\mid21$.
\end{center}
This contradiction completes the proof of Proposition \ref{proposition:PGL27-gorenstein}.

\subsection{Non-Gorenstein models for $\PGL_2(\FF_7)$}
\label{subsection:PGL27-non-gorenstein}

We keep the notation $G=\PGL_2(\FF_7), A=\PSL_2(\FF_7)$, and we choose an involution $\tau\in G\setminus A$. In this section we prove the following Proposition, which finishes the proof of \hyperref[auxtheorem]{Characterization Theorem}.

\begin{proposition}
    Let $X$ be a rational non-Gorenstein $G\QQ$-Fano threefold for a finite group $G\subset\Aut(X)$ such that $G\simeq\PGL_2(\FF_7)$. Then $X$ is $G$-birational to the threefold \eqref{equation:W}.
\end{proposition}

We follow the construction of Section \ref{section:A6-ext-non-Gorenstein}, recording the changes required for $G$ and referring to the corresponding arguments in \cite[Section 7]{CheltsovPinardinProkhorov}. Suppose such a threefold exists for $G\simeq\PGL_2(\FF_7)$ and is not $G$-birational to the threefold \eqref{equation:W}. By Proposition \ref{proposition:PGL27-gorenstein} every rational $G\QQ$-Fano threefold in its $G$-equivariant birational class is non-Gorenstein.

Denote by $\Sigma$ the set of non-Gorenstein points of $X$. The basket computation in Section \ref{section:A6-ext-non-Gorenstein}, using the stabilizers in $G$ and the fact that its subgroup $\mathrm{F}_7:=\mumu_7\rtimes\mumu_6$ does not fix points, gives that $N=\vert\Sigma\vert\leqslant14$, every point in $\Sigma$ is of type $\frac{1}{2}(1,1,1)$ and $\Sigma$ is either a union of $G$-orbits of length $2$ or one orbit of length $14$; cf. \cite[Proposition 2.7.2, Corollary 2.7.5]{CheltsovPinardinProkhorov}. 

To prove that $\vert-K_X\vert\neq\varnothing$, the same argument as in \cite[Section 7.2]{CheltsovPinardinProkhorov} applies. Indeed, if $\vert-K_X\vert=\varnothing$, Riemann-Roch gives
\begin{center}
    $N=14,\quad (-K_X)^3=1,\quad h^0(X,-2K_X)=4$.
\end{center}
Thus $A$ acts trivially on $H^0(X,-2K_X)$, and for a general member $S\in\vert-2K_X\vert$ the restriction map makes $H^0(S,2K_S)$ an $A$-equivariant quotient of this space. The proofs of \cite[Lemmas 2.7.7--2.7.8]{CheltsovPinardinProkhorov} show that $S$ has Du Val singularities. Then the same bicanonical argument applies on its minimal resolution and produces a contradiction.

The positive-dimensionality, mobility and K3-section arguments in Section \ref{section:A6-ext-non-Gorenstein} now apply, with \cite[Corollary 2.5.10]{CheltsovPinardinProkhorov} replacing the corresponding statement for $\mathfrak{A}_6$; compare \cite[Proposition 2.7.11, Corollaries 2.7.12--2.7.16]{CheltsovPinardinProkhorov}. Applying the argument in Lemma \ref{lemma:A6-ext-non-Gorenstein-A6-inv} to a $G$-invariant pencil in $\vert-K_X\vert$ we see that $\vert-K_X\vert$ contains at most one $A$-invariant member. Consequently, $n=\dim\vert-K_X\vert\geqslant5$. The anticanonical construction in Section \ref{section:A6-ext-non-Gorenstein}, as in \cite[Lemma 2.7.19]{CheltsovPinardinProkhorov}, gives
\begin{center}
    $\pi:Y\longrightarrow X,\qquad$ $\phi:Y\longrightarrow\overline{X},\qquad-K_Y=\phi^\ast(-K_{\overline{X}})$,
\end{center}
where $\phi$ is birational, $\pi$ blows up $\Sigma$, $\overline{X}$ is a canonical Gorenstein Fano threefold anticanonically embedded in $\PP^n$, and $E_i\simeq\PP^2$, $\phi\vert_{E_i}:E_i\overset{\sim}{\rightarrow}\overline{E}_i$ identifies each exceptional divisor with a plane. The representation arguments used in Subsection \ref{subsection:PGL27-gorenstein}, together with the genus-four exclusion of Subsection \ref{subsection:irrationality}, show that $\overline{X}$ is cut out by quadrics and $g:=g(\overline{X})\in\{6,10,11,\ldots,37\}$.

\begin{lemma}
    \label{lemma:PGL27-small}
    The morphism $\phi$ is small.
\end{lemma}

\begin{proof}
    We follow the arguments of Section~\ref{section:A6-ext-non-Gorenstein}. Suppose that $\phi$ is not small, and let $F_1,\ldots,F_k$ be its exceptional divisors. First, one observes that any exceptional divisor $F_j$ is mapped to a curve under $\phi$, and these curves satisfy
    $$\sum_{i=1}^k \deg\phi(F_i)\leqslant19.$$
    Then this inequality is used to prove that $X$ has exactly $N=14$ non-Gorenstein points corresponding to $\frac{1}{2}(1,1,1)$ quotient singularities which form a unique $G$-orbit, and every $\phi(F_i)$ has to be a line. As $\mathrm{r}(\overline{X})^G=1$ we compute that
    $$\overline{D}:=\sum_{i=1}^{14}\overline{E}_i\sim_{\QQ}\frac{7}{g-1}(-K_{\overline{X}}),$$
    so in particular $\overline{D}$ is ample and then connected in codimension one. Hence, there are some $i,j\in\{1,\ldots,14\}$ such that  $\overline{E}_i\cap\overline{E}_j$ is a line. Moreover, one can observe that these lines correspond to one of the lines $\phi(F_1),\ldots,\phi(F_{k})$. Then, the same statement of Lemma~\ref{lemma:A6-ext-non-gor-lc} holds and $(\overline{X},\overline{D})$ is not log canonical at the general point of each of these lines. Using this fact, a contradiction can be obtained using the argument at the end of Subsection~\ref{subsection:A6-ext-anticanonical-map-not-contracting-divisors}.
\end{proof}

Thus, $\overline{X}$ has terminal Gorenstein singularities. In particular, it follows from Lemma~\ref{lemma:A6-ext-index} that $\iota(\overline{X})=1$. Then, considering a $\QQ$-Gorenstein smoothing $V$ of $\overline{X}$ as in the beginning of Section~\ref{section:A6-ext-gorenstein}, from the classification of smooth Fano threefolds we obtain the bound 
$$\vert\Sing(\overline{X})\vert\leqslant20+h^{1,2}(V)-\rho(V)\leqslant29.$$
The same argument as in the proof of Lemma \ref{lemma:2-planes-through-point-A6-ext} shows that at most two planes among $\overline{E}_1,\ldots,\overline{E}_N$ pass through any point. 

Choose $X$ with $\dim\vert-K_X\vert$ maximal in its equivariant birational class, as in \eqref{equation:maximal}. The proof of \cite[Lemma 2.7.39]{CheltsovPinardinProkhorov} implies that $\phi$ is not an isomorphism: the other extremal contraction after blowing up one $G$-orbit cannot be of fiber type, and a Gorenstein $G\QQ$-Fano target is excluded by Subsection \ref{subsection:PGL27-gorenstein}. An $A$-invariant plane $E_i$ is disjoint from $\operatorname{Exc}(\phi)$ by the proof of \cite[Lemma 2.7.40]{CheltsovPinardinProkhorov}: otherwise it and its disjoint $\tau$-transform would contain at least $42$ singular points. Every curve $C$ contracted by $\phi$ meets $\sum E_i$, by
$$0=-K_Y\cdot C=\pi^\ast(-K_X)\cdot C-\frac{1}{2}\sum E_i\cdot C,$$
and then we see that $\Sigma$ cannot consist entirely of orbits of length $2$. Therefore, $N=14$ and $\Sigma$ is a $G$-orbit. 

Since $\phi$ is not an isomorphism and $G$ permutes the planes transitively, every $\overline{E}_i$ meets $\operatorname{\Sing}(\overline{X})$ and its stabilizer $\mathfrak{S}_4$ has no orbits of length one or two on it. The counting argument of Section \ref{section:A6-ext-non-Gorenstein} therefore gives $\vert\operatorname{Sing}(\overline{X})\vert\geqslant14\cdot 3/2=21$. Comparing with the smoothing bound as in \cite[Corollary 2.7.43]{CheltsovPinardinProkhorov} yields $g\in\{6,10,11,12\}$. 

Since $\mathrm{r}(Y)^G=2$, the equivariant two-ray game of \cite[Section 7.6]{CheltsovPinardinProkhorov} gives a link
$$
\xymatrix{
X & & \ar[ll]_\pi \ar@{-->}[rr] Y & & Y^\prime \ar[rr]^{\pi^\prime} & & X^\prime 
},
$$
where $Y\dashrightarrow Y^\prime$ is a composition of flops and $\pi^\prime$ is a divisorial contraction to a non-Gorenstein rational $G\QQ$-Fano threefold $X^\prime$. By maximality we have
$$h^0(X^\prime,-K_{X^\prime})=h^0(Y^\prime,-K_{Y^\prime})=h^0(Y,-K_{Y})=h^0(X,-K_{X}),$$
and the arguments in the proof of \cite[Lemma 2.7.44]{CheltsovPinardinProkhorov} identify $\pi^\prime$ as the blow-up of non-Gorenstein points of $X^\prime$. Since $X^\prime$ also satisfies assumption \eqref{equation:maximal}, the previous arguments apply to $X^\prime$ so this threefold has exactly $14$ non-Gorenstein points forming a unique $G$-orbit. 

Put $H=-K_Y$, $E=\sum_{i=1}^{14} E_i$, and write the strict transform on $Y$ of the $\pi^\prime$-exceptional divisor as $E_Y^\prime\sim_\QQ aH-bE$, with $a>0$. Intersections with $H$ are preserved under flops, so
\begin{center}
    $H^3=2g-2,\quad$ $H^2\cdot E=14,\quad$ $H\cdot E^2=-28$,
\end{center}
and then we obtain
\begin{center}
    $14=(2g-2)a-14b,\qquad$ $-28=(2g-2)a^2-28ab-28b^2$. 
\end{center}
Hence $b=1$ and $a=\frac{14}{g-1}$. Pushing forward to $\overline{X}$ we get
$$D:=\phi_\ast(E+E_Y^\prime)\sim_\QQ \frac{14}{g-1}(-K_{\overline{X}}).$$
Then $D$ is Cartier, and restrincting to a plane $E_i$ gives $\frac{14}{g-1}\in\mathbb{Z}$, which is not true for any possible value of $g$.

This contradiction proves that $X$ is $G$-birational to $W$ and therefore we conclude the proof of the \hyperref[auxtheorem]{Characterization Theorem}.

\bibliography{cremona}
\bibliographystyle{alpha}
\nocite{*}

\end{document}